\documentclass[leqno, 10pt]{amsart}
\usepackage{amsmath, amssymb, amsthm, graphics}

\theoremstyle{plain}
\newtheorem{theorem}{Theorem}[section]
\newtheorem{lemma}{Lemma}[section]
\newtheorem{corollary}{Corollary}[section]
\newtheorem{proposition}{Proposition}[section]

\theoremstyle{definition}
\newtheorem{definition}{Definition}[section]
\newtheorem{remark}{Remark}[section]
\newtheorem{example}{Example}[section]

\begin{document}

\title[Boundary behavior of arbitrary and meromorphic functions]
{Normality and boundary behavior of arbitrary and  meromorphic functions along
simple curves and applications}

\author{\v{Z}arko Pavi\'{c}evi\'{c}}
\address{Faculty of Natural Sciences and Mathematics, University of Montenegro,
Cetinjski put b.b. 81000 Podgorica, Montenegro}
\email{zarkop@ac.me}

\author{Marijan Markovi\'{c}}
\address{Faculty of Natural Sciences and Mathematics, University of Montenegro,
Cetinjski put b.b. 81000 Podgorica, Montenegro}
\email{marijanmmarkovic@gmail.com}



\begin{abstract}
We  establish the  theorems that give necessary and  sufficient conditions for
an arbitrary function defined in the unit disk of complex plane in order    to
has boundary  values along  classes of equivalencies  of simple  curves.   Our
results       generalize the  well--known theorems on asymptotic   and angular
boundary behavior of     meromorphic  functions (Lind\"{o}lf,  Lehto--Virtanen,
and Seidel--Walsh type theorems).      The results are applied to the study of
boundary behavior  of  meromorphic functions along curves using  $P-$sequences,
as well as in the  proof  of the  uniqueness theorem similar to \v{S}aginjan's
one.      Constructed  examples of  functions  show that the results cannot be
improved.
\end{abstract}

\maketitle

\tableofcontents

\section{Introduction}
In  this   paper  we  study some  problems of the Theory of cluster sets,   a
theory  which   is  developed in the second half of the twentieth century. It
is  believed  that   the      first result  of  this theory were  obtained by
Sohotsky      \cite{SOKHOTSKY}, and independently by Cazorati \cite{CASORATI}
in  1868th,   and Weierstrass  \cite{WEIERSTRASS} in 1876th, which   is known
in     literature   as the Theorem on   essential  singularity of    analytic
functions (see \cite{SHABAT.BOOK} p. 123).       Fundamentals   of the Theory
of     cluster   sets                   are      presented  in     monographs
\cite{COLLINGWOOD.LOHWATOR.BOOK}, \cite{NOSHIRO.BOOOK}, \cite{SOKHOTSKY}, and
in the more   recent   survey  paper \cite{LOHWATER}.

The main objects     of   research  in this paper is  the asymptotic behavior
of  meromorphic   functions along a simple curve  ending     in   a  boundary
point   of       the domain of functions.            We emphasize that a very
productive area  of  investigation   are domains of the hyperbolic type, i.e.,
domains  where one  may define the hyperbolic metric.

One of the classical  results of The theory of  cluster  sets  related to the
asymptotic     behavior  is  the theorem of Lindel\"{o}f on angular  boundary
values   of analytic    functions \cite{LINDELOF} (or,    see \cite{VALIRON}).
Further interesting  results on   the boundary behavior of analytic functions
along  simple curves  were obtained by Seidel  \cite{SEIDEL.TAMS}  and Seidel
and Walsh    \cite{SEIDEL.WALSH.TAMS}      (see also \cite{LOHWATER}).  Lehto
and Virtanen's  result from  \cite{LEHTO.VIRTANEN.ACTA}, which is a  transfer
of  the     results  of Lindel\"{o}f and Seidel and   Walsh to   the class of
normal meromorphic   functions in the unit disc,  the class  usually  denoted
by $N$, prompted     a  further intensive  research in  the Theory of cluster
sets. These   investigations   were also  related to the boundary behavior of
functions  along sequences of  points on the one hand, and on  the   boundary
behavior of harmonic,           subharmonic, continuous functions, and normal
quasiconformal and equimorphic mappings along (non--)tangential simple curves
(see References).

While most of these papers concern the boundary      properties of functions
along simple   curves which are at the finite Fr\'{e}shet distance or finite
Hausdorff distance  (see eg. \cite{DEVYATKOV}),   in  this paper we define a
relation  of equivalence in  the  family   of all simple  curves in the unit
disc which  terminate in the same point on the boundary,  and  study     the
boundary  behavior of functions along classes of equivalence. We also  offer
an example of two simple curves ending in  a   point of the boundary of  the
unit   disc   which       belong to  the   same equivalence class, such that
their Fr\'{e}shet distance  is infinite. This the content of Lemma \ref{LE4}.
Thus,     our results in the paper are generalization of some  known results.
Namely,  using the mentioned relation of equivalence we prove   the  theorem
that give necessary and sufficient conditions for an      arbitrary function
defined in the unit disk to has a curvilinear boundary value    (see Theorem
\ref{TH1}).  This theorem is used in proof of     Theorem \ref{TH2},   which
shows that for an arbitrary function in the unit   disk holds an    analogue
of  Theorem 1 in \cite{LEHTO.VIRTANEN.ACTA}     concerning  the  meromorphic
functions.    As  follows   from our  Theorem \ref{TH1}, the normality along
simple curves is a    necessary  condition for the existence  of curvilinear
boundary value  of functions. In Section  5      we study the  normality and
boundary  behavior of    meromorphic  functions using the  $P-$sequences. We
emphasis         that the $P-$sequences  provide   necessary  and sufficient
conditions for meromorphic  functions to be normal (see  \cite{GAVRILOV.SBORNIK.67},
\cite{PAVICEVIC.NEW.ZELAND}, \cite{GAVRILOV.SBORNIK.71}, \cite{GAUTHIER.NAGOYA}).
Further, in Section 6 we prove theorems that give necessary and  sufficient
conditions in order  that a  meromorphic function in the  unit disk has   a
curvilinear boundary value (see Theorems \ref{TH14} and  \ref{TH15}). These
theorems are analogous to the theorems 2, 2', 4 and 5 in Lehto and Virtanen
work \cite{LEHTO.VIRTANEN.ACTA}. While Theorem 2, 2', 4 and 5  of Lehto and
Virtanen concern  the class  $N$ of normal meromorphic function in the unit
disk, the results of Theorems \ref{TH14}  and \ref{TH15} are related to the
class of normal  meromorphic    functions along a simple curve ending in  a
boundary point of the unit disc.   These classes  are wider than the  class
$N$; that will be showed  by Examples \ref{EX1} and \ref{EX2} in  Section 7.
Our results are applied in Section 8 in order  to derive Theorem \ref{TH16},
which shows  that the domain  along which there is a single boundary  value
of meromorphic functions in $N$ from  Theorems 2, 2', 4 and 5 in Lehto  and
Virtanen       work  \cite{LEHTO.VIRTANEN.ACTA} can spread in the  case  of
simple curves which are tangent to the boundary of the unit disc.   However,
one cannot obtain an extension by using the method of  Lehto  and Virtanen.
Finally,  our results are used to show the uniqueness theorem of \v{S}aginjan \cite{SAGINJAN}
which is related to the class  of   boundary  analytic
functions, and it's  generalization    to the class  $N$ that was    obtained
by Gavrilov \cite{GAVRILOV.SAGINJAN}. We prove a similar result for the class
of  meromorphic     functions       in    the unit disk that are normal along
non--tangential simple curves.

\section{Notations}
By $\mathbb{D}$ we will denote the open unit disk $\{z:|z|<1\}$ in the complex
plane $\mathbb{C}$ and   by   $\Gamma$  the  unit       circle   $\{z:|z|=1\}$.

Let
\begin{equation*}
d_{ph}(z,w)     = \left|\frac{z-w}{1-z\overline{w}}\right|\quad\text{and}\quad
d_{h}(z,w)=\log\frac{1+d_{ph}(z,w)}{1-d_{ph}(z,w)}
\end{equation*}
stands for the pseudo--hyperbolic distance and the hyperbolic distance between
$z,\,  w\in\mathbb{D}$, respectively. It is  well known that $d_h$ is a metric
in  the unit disc, and that $(\mathbb{D},d_h)$ is the Poincar\'{e} disc  model
for  the Lobachevsky geometry.
Furthermore, denote by
\begin{equation*}
d_S (z,w) = \left\{
\begin{array}{ll}
\frac{2|z-w|}{\sqrt{1+|z|^2}\sqrt{1+|w|^2}}, & \hbox{$z,\, w\in\mathbb{C}$},\\
\frac {2} {\sqrt{1+|z|^2}},              &\hbox{$z\in\mathbb{C},\, w=\infty$}.
\end{array}
\right.
\end{equation*}
the    spherical metric in the extended complex plane $\overline{\mathbb{C}}=
\mathbb{C}\cup\{\infty\}$ (Riemann sphere).

For $r>0$    we     denote   by ${D}(r)=\{|z|< r\}$ the standard open disc in
$\mathbb{C}$ with  centre in $0$ and radius $r$.     For $z\in\mathbb{D}$ let
$D_{h}(z,r)=\{w\in\mathbb{D}:D_{h}(z,w)<r\}$ be a     disc in the  hyperbolic
metric. Let $D_S(w,r),\, w\in\overline{\mathbb{C}}$ denote a disc on      the
Riemann sphere. For $r'\in (0,1)$ the  set $D_{ph}(z,r') = \{w\in \mathbb{D}:
d_{ph}(z,w)<r'\}$ stands  for the pseudo--hyperbolic  disc with centre in $z$
and    pseudo--hyperbolic  radius $r'$.   In  a similar manner one introduces
the closed  discs in these   metrics.   It  is  straightforward to show  that
\begin{equation}\label{EQ1}
\overline{D}_{h}(z,r')=\overline{D}_{ph}(z,r)\quad\text{with}\quad r\in[0,1),\,
r'=\log\frac{1+r}{1-r}\in[0,\infty).
\end{equation}

The group of all M\"{o}bius transforms of $\mathbb{D}$ onto itself (conformal
automorphisms of the unit disc) will be denoted by $\mathcal {M}$. A function
$f:\mathbb{D}\rightarrow\overline{\mathbb{C}}$ is normal  in  $\mathbb{D}$ if
the family $\{f\circ \varphi:\varphi\in\mathcal{M}\}$ is a normal   family in
the sense of Montel, i.e., if any sequence of this family  has a  subsequence
which is convergent in local topology of $\mathbb{D}$ (uniformly on   compact
subsets of $\mathbb{D}$). All     sequences of functions (or numbers) we mean
are convergent in above metrics (if they are convergent).    Particulary, the
uniform convergence on  compact   subsets  of      the disc $\mathbb{D}$ of a
sequence of functions  $\{f_n:\mathbb{D}\rightarrow\overline{\mathbb{C}}:n\in
\mathbb{N}\}$ to a function $f:\mathbb{D}\rightarrow\overline{\mathbb{C}}$ we
mean    in  the     metrics  of      spaces        $(\mathbb{D},d_{ph})$  and
$(\overline{\mathbb{C}},d_S)$,   or what is the same, in $(\mathbb{D},d_{h})$
and $(\overline{\mathbb{C}},d_S)$, as follows from \eqref{EQ1}.

For $w\in  \mathbb{D}$    let $\varphi_w\in\mathcal {M}$ be defined       by
\begin{equation*}
\varphi_w(z) = \frac{z+w}{1+z\overline{w}}.
\end{equation*}
If $f:  \mathbb{D}\rightarrow\overline{\mathbb{C}}$  is any function, we will
use the notation $f_w$ for $f\circ\varphi_w:\mathbb{D}\rightarrow   \overline
{\mathbb{C}}$, where $\varphi_w$  is defined above. In   the sequel  we  will
consider the following type of family of functions $\{f_n=f\circ \varphi_n\}$,
where $\varphi_n =  \varphi_{w_n}$ and $\{w_n\}$   is a sequence of points in
$\mathbb{D}$ such that $\lim_{n\rightarrow\infty} w_n= e^{i\theta}\in\Gamma$.

The set $C(f,A,e^{i\theta})  =\{w\in\overline{\mathbb{C}}:\text{there exist a
sequence}\, \{z_n\}\subseteq A,\, \lim_{n\rightarrow\infty}z_n=e^{i\theta}\in
\Gamma\,    \text{such that}\, \lim_{n\rightarrow\infty} f(z_n) =w\}$  is the
cluster  set for the function  $f:\mathbb{D}\rightarrow\overline{\mathbb{C}}$
in the  point $e^{i\theta}$   along the set $A$  whose closure in $\mathbb{D}
\cup\Gamma$     contains         $e^{i\theta}$.        It may be checked that
$C(f,A,e^{i\theta})$  is closed.

All curves which appear  in the text  we mean lie in $\mathbb{D}$, are simple
and terminate in a point $e^{i\theta}\in\Gamma$. Let $\gamma$   be a such one
curve. The  set
\begin{equation*}
\Delta_r  \gamma=\bigcup_{z\in\gamma}\overline{D}_{ph}(z,r),
\end{equation*}
where $r\in [0,1)$, is called  a curvilinear angle  along  the curve $\gamma$
with deflection $r$  and with vertex in $e^{i\theta}$. Particulary, for $r=0$
we have $\Delta_0\gamma = \gamma$. Regarding \eqref{EQ1}, we have
\begin{equation*}
\Delta_r \gamma =
\bigcup_{z\in\gamma}\overline{D}_{ph}(z,r)=\bigcup_{z\in\gamma}\overline{D}_{h}
\left(z,r'\right)\quad\text{for all}\quad r\in [0,1).
\end{equation*}
For the curvilinear angle $\Delta_r\gamma$ we will sometimes use the notation
$\Delta_{r'} \gamma$.  Although this is not fully precise, we believe  that a
misunderstanding  will not occur.

\begin{example}
If    $\gamma$ is  the radius  of the disc $\mathbb{D}$  with one endpoint in
$e^{i\theta}\in\Gamma$, then the curvilinear angle  along       $\gamma$ with
deflection $r$ is   the  domain   bounded by arcs of two hyper--cycle    with
endpoints in $e^{i\theta}$ and $e^{-i\theta}$ and by the arc of    the circle
$|z|=r$.    That     curvilinear  angle        we call a hyperbolic angle. If
$h(\theta,\alpha_1)$ and $h(\theta,\alpha_2),\, -\frac{\pi}2<\alpha_1<\alpha_2
<\frac\pi2$, are chords  of the disc  $\mathbb{D}$  which     with the radius
$r_\theta$  of $\mathbb{D}$ with one endpoint at $e^{i\theta}$    form angles
$\alpha_1$ and  $\alpha_2$, then the sub--domain  of  $\mathbb{D}$ bounded by
these chords and  the circle $\left\{z:|z-e^{i\theta}|= r\right\}$     is the
Stolz angle with  vertex in  $e^{i\theta}$.          For any  Stolz  angle in
$\mathbb{D}$ with vertex  in  $e^{i\theta}$   there exists a hyperbolic angle
$\Delta_r\gamma$ which is contained in  it;   we have  also the converse: any
hyperbolic     angle  $\Delta_r\gamma$ contains a Stolz angle in $\mathbb{D}$
with vertex  in  $e^{i\theta}$. If   $\gamma$    is an arc of the horo--cycle
$\{z:|z-\frac {e^{i\theta}}2|=\frac 12\}$  with  endpoints  $0$ and $e^{i\theta}$,
then the curvilinear angle $\Delta_r \gamma$ is the domain    bounded by arcs
of two horo--cycles  which contain $e^{i\theta}$. That      curvilinear angle
$\Delta_r\gamma$ we call  the horo--cyclic angle (see \cite{FUKS}).
\end{example}

At the end of this section we recall the known definition of the Fr\'{e}shet
distance between two curves.  For       curves $\gamma_1$ and $\gamma_2$ the
Fr\'{e}shet distance between them is
\begin{equation*}
d_F (\gamma_1,\gamma_2)  =\inf_\varphi \sup_{z\in\gamma_1} d_h(z,\varphi(z));
\end{equation*}
the infimum  is taken  among  all homeomorphisms $\varphi:\gamma_1\rightarrow
\gamma_2$.

\section{Preliminaries}
The following lemma is straightforward and therefore we omit a proof.

\begin{lemma}\label{LE1}
For all $r\in [0,1)$ and $w\in\mathbb{D}$ we have
\begin{equation*}
\overline{D}_{ph}(w,r)=\varphi_w (\overline{D}(r)).
\end{equation*}
Thus,
\begin{equation*}
\Delta_r \gamma  = \bigcup_{w\in\gamma}\overline{D}_{ph}(w,r) = \bigcup_{w\in
\gamma}\varphi_w (\overline{D}(r)).
\end{equation*}
\end{lemma}

\begin{definition}\label{DEF1}
Let $\gamma_1$      and  $\gamma_2$  be two curves with a    same endpoint in
$\Gamma$.     If  $\gamma_2\subseteq \Delta_r\gamma_1= \bigcup_{w\in\gamma_1}
\overline{D}_{ph}(w,r)$  for some   $r\in (0,1)$,             we say that the
pseudo--hyperbolic distance  between   $\gamma_2$ and $\gamma_1$ is less then
$r$. If  instead of the pseudo--hyperbolic distance we use   the   hyperbolic
distance,  then we say that the hyperbolic distance between    $\gamma_2$ and
$\gamma_1$ is less then $r'  =  \log\frac{1+r}{1-r}\in (0,\infty)$ (regarding
\eqref{EQ1}). In this case we   will  simply   say that  the distance between
curves if finite  (see the following lemma);  if  this is  not  the case,  we
say that the  distance is infinite.
\end{definition}

\begin{example}\label{RE2}
Hyper--cycles   in    $\mathbb{D}$ which terminate in a point $e^{i\theta}\in
\Gamma$ are simple  curves   such   that the distance between any of them  is
finite.   The  same is true for horo--cycles  in $\mathbb{D}$ which terminate
in $e^{i\theta}$.  However, the distance between    any      hyper--cycle and
any  horo--cycle  both terminating in $e^{i\theta}$     is  infinite.
\end{example}

The following definition introduces a relation $\sim$  in the family  of all
curves in $\mathbb{D}$ which  terminate in a same  point in  $\Gamma$.

\begin{definition}\label{DEF2}
Let $\gamma_1$  and  $\gamma_2$ be curves in the disc $\mathbb{D}$ ending in
$e^{i\theta}\in \Gamma $. We write $\gamma_1\sim\gamma_2$ if there exist  $r
\in (0,1)$ such that   $\gamma_1\subseteq \Delta_{r} \gamma_2$, i.e., if the
distance between  $\gamma_1$ and  $\gamma_2$    is finite.
\end{definition}

\begin{lemma}\label{LE2}
The relation     $\sim$ is an relation of equivalence in the family of  all
curves in the disc $\mathbb{D}$  with the same endpoint  in  $\Gamma$.  The
class of equivalence for   a curve $\gamma$   will be denoted by $[\gamma]$.

In   order  to     establish the symmetry property of $\sim$ we will use the
following  assertion. Let the  hyperbolic  between $\gamma_1$ and $\gamma_2$
be less   then $r\in (0,\infty)$, then  we have: For all $r_1\in (0,\infty)$
there exist     $r_2\in (0,\infty)$ such that $\Delta_{r_1}\gamma_1\subseteq
\Delta_{r_2}\gamma_2$, and   for all $r_2\in (0,\infty)$ there exist $r_1\in
(0,\infty)$ such that   $\Delta_{r_2}\gamma_2\subseteq \Delta_{r_1}\gamma_1$.
\end{lemma}

\begin{proof}
We      will firstly prove the assertion. We  will proof the first statement,
since the second follows immediately from the first one.

From Definition \ref{DEF1}    it follows $\gamma_1\subseteq\Delta_r\gamma_2=
\bigcup_{w\in\gamma_2}\overline{D}_h(w,r)$. Let  $r_2=r_1+r$.  We  will show
$\Delta_{r_1}\gamma_1\subseteq\Delta_{r_2}\gamma_2$. Let $z\in  \Delta_{r_1}
\gamma_1$. There exists  $w_1\in \gamma_1$ such that $z\in\overline{D}_h(w_1,
\gamma_1)$. Since $\gamma_1\subseteq \Delta_r \gamma_2$, there exist $w_0\in
\gamma_2$   such that $w_1\in D_h(w_0,r)$. Using the triangle inequality, we
obtain
\begin{equation*}
d_h(z,w_0)\le  d_h(z,w_1)+ d_h (w_1 ,w_0) < r_1+r=r_2;
\end{equation*}
thus   $z\in D_h(w_0,r_2)$, i.e.,  $z\in \Delta_{r_2} \gamma_2$.     We have
proved  $\Delta_{r_1}\gamma_1\subseteq \Delta_{r_2}\gamma_2$.

Let     us now establish the that $\sim$ is a relation of equivalence. It is
clear that   $\gamma\subseteq\Delta_r\gamma$ for all $r\in (0,\infty)$, what
means   that $\sim$ is reflexive. If  $\gamma_1\sim\gamma_2$, then $\gamma_1
\subseteq\Delta_r\gamma_2$; from   the assertion it follows that there exist
$r'$ such that  $\gamma_2\subseteq\Delta_{r'}\gamma_1$, i.e., $\gamma_2 \sim
\gamma_1$.  Thus,      from $\gamma_1\sim\gamma_2$ it follows $\gamma_2 \sim
\gamma_1$. We have proved that $\sim$ is a symmetry relation. It  remains to
establish      the  transitivity of $\sim$.   Let $\gamma_1\sim\gamma_2$ and
$\gamma_2\sim\gamma_3$. It follows  $\gamma_2 \subseteq\Delta_{r_1}\gamma_1$
and $\gamma_3\subseteq\Delta_{r_2}\gamma_2$       for some $r_1,\, r_2\in (0,
\infty)$. As   in the assertion, one may  prove   $\gamma_3\subseteq\Delta_s
\gamma_1$ for  $s = r_1 + r_2$. Thus, $\gamma_1\sim \gamma_3$.
\end{proof}

\begin{remark}
For any point $e^{i\theta}\in\Gamma$   there exist infinity many  classes of
equivalence for the  relation $\sim$.     Namely,  if curves  $\gamma_1$ and
$\gamma$ have a different order of contact   on  the    unit circle $\Gamma$
in the point  $e^{i\theta}$, then   $[\gamma_1]\ne  [\gamma_2]$. For example,
all curves in  the $\mathbb{D}$ which terminate in  $e^{i\theta}$ that   are
not tangent to        $\Gamma$ in $e^{i\theta}$ belong  to the same class of
equivalences. All  horo--cycles tangent   on   the boundary $\Gamma$  in the
point $e^{i\theta}$ belong to the same class  of equivalences, etc.
\end{remark}

\begin{lemma}\label{LE4}
If the Fr\'{e}shet  distance between curves   $\gamma_1,\, \gamma_2\subseteq
\mathbb{D}$ ending in the same point $e^{i\theta}\in\Gamma$  is finite, then
$\gamma_1\sim\gamma_2$. The converse does not hold.
\end{lemma}

\begin{proof}
If      the Fr\'{e}shet distance between $\gamma_1$ and $\gamma_2$ is finite,
then from     the definition of the Fr\'{e}shet distance and from Definition
\ref{DEF2} immediately  follows   that $\gamma_1 \sim\gamma_2$.   Namely, if
$d_F(\gamma_1, \gamma_2)<r$ for some $r\in (0,\infty)$, then        for some
homeomorphism   $\varphi:\gamma_1\rightarrow\gamma_2$ holds $d(z,\varphi(z))
<r$ for all $z\in \gamma_1$. This clearly        implies  $\gamma_2\subseteq
\Delta_r \gamma_1$,  i.e.,    $d(\gamma_1, \gamma_2)\le r<\infty$. Thus,  in
view of Definition \ref{DEF2}   we have  $\gamma_1\sim\gamma_2$.

In  order to prove the second statement of this lemma, we  will construct  an
example of two curves   $\gamma_1$     and $\gamma_2$ such that $\gamma_1\sim
\gamma_2$, but     the Fr\'{e}shet distance between $\gamma_1$ and $\gamma_2$
is infinite.

Let    $\gamma_1$ be a radius od the unit disc $\mathbb{D}$ with one endpoint
in $e^{i\theta}$.   We will show that for any $r\in (0,\infty)$ there   exist
a curve   $\gamma_2\subseteq\Delta_r\gamma_1$  that the Fr\'{e}shet  distance
between $\gamma_1$ and $\gamma_2$ is not finite.

{\it Construction of $\gamma_2$}. Let us chose points  $z_1,\, w_1,\, z_2,\,
w_2, ..., z_n,\, w_n,...\in\gamma_1$ such that $d_h(z_n,z_{n+1})\rightarrow
\infty$ as $n\rightarrow \infty,\, d_h(z_n,w_{n+1})\le 1$  for all $n\ge 1$,
and that the  order of crossing  of $\gamma_1$ thought the preceding points
is  as    they appear in the sequence. Regarding the way we selected points
$z_1,\, w_1,\, z_2,\, w_2,...,z_n,\, w_n,...$ it follows that $d_h (z_{n +1},
w_n ) \rightarrow \infty$ as $n\rightarrow\infty$.  For  $\gamma_2$ we will
take any curve  in $\Delta_r\gamma_1$ which contains the  preceding  points
in the following order:
\begin{equation*}
z_1,\, z_2,\, w_1,\, z_3,\, w_2,\, z_4,\dots,\, z_n,\, w_{n-1},\ z_{n+1},\dots
\end{equation*}
We will  show now that the Fr\'{e}shet distance   between curves $\gamma_1$
and $\gamma_2$ is not finite. Assume that  $\varphi$    is   an  arbitrary
homeomorphism between  $\gamma_2$ and $\gamma_1$.   Let $z_1,\, z_2,\, z_3,\dots,\, z_n,...$
($w_1,\, w_2,\, w_3,\dots,w_n,\dots$ which  belong  to $\gamma_2$)      be
corespondent to  $t_1,\, t_2,\, t_3,\dots,t_n,\dots$ (i.e., $s_1,\, s_2,\, s_3,\dots,s_n,\dots$
in  $\gamma_1$) via the homeomorphism $\varphi$.  The schedule of crossing of $\gamma_1$
thought   the  preceding points   is:
\begin{equation*}
t_1,\, t_2,\, s_1,\, t_3,\, s_2,\, t_4, \dots, t_n,\ s_{n-1},\, t_{n+1},\, s_n,\dots,
\end{equation*}
i.e., the curve across the point   $t_{n+1}$   and then  $s_n$.  Since
$d_h (z_{n+1},w_n) \rightarrow \infty$ as $n\rightarrow\infty$, for any $D>0$
there exist an integer $n_0\ge  1 $ such that for every $n\ge n_0$    holds
$d_h (z_{n+1}, w_n ) > 2D +1$. However,  if  the Fr{\'e}shet  distance
between $\gamma_1$ and $\gamma_2$ would be  bounded by $D$, then we will
have $d_h(z_{n+1},\varphi(z_{n+1})) =  d_h (z_{n+1},t_{n+1})\le D$ and
$d_h(w_n ,\varphi(w_n ))=  d_h (w_n , s_n )\le D$   for every    $n\ge n_0$.
In view of the relation $d_h(z_{n+1},w_n ) > 2D +1$ it follows   that the
curve $\gamma_1$ first across the  point $s_n$ and then $t_{n+1}$. This is
the contradiction.
\end{proof}

\section{Curvilinear boundary behavior of arbitrary functions}

\begin{theorem}\label{TH1}
Let  $f:\mathbb{D}\rightarrow\overline{\mathbb{C}}$ be an arbitrary function,
let $\{w_n\}\subseteq\mathbb{D}$ be a sequence such that $\lim_{n\rightarrow
\infty} w_n=e^{i\theta}$, and let $c\in\overline{\mathbb{C}}$. The following
two conditions are equivalent:
\begin{enumerate}
\item the sequence $\{f_n=f\circ \varphi_n\}$, where $\varphi_n=\varphi_{w_n}$,
is convergent in the local topology of $\mathbb{D}$ to the constant function
$c$;
\item  for any compact subset $K\subset\mathbb{D}$ holds
\begin{equation*}
C \left(f,\bigcup_{n\in\mathbb{N}} \varphi_{n}(K), e^{i\theta}\right) =\{c\}.
\end{equation*}
\end{enumerate}
\end{theorem}

\begin{proof}
{\it (1) implies (2)}. Let $K$ be any compact subset of $\mathbb{D}$. Since
$\lim_{n\rightarrow\infty} w_n=e^{i\theta}$, it  follows that $e^{i\theta}$
is the  single   point od adherence in $\Gamma$ for the set  $\bigcup_{n\in
\mathbb{N}}\varphi_{w_n}(K)$. Since     $\{f_n = f\circ \varphi_{w_n}\}$ is
uniformly         convergent to the  constant $c$ on the  compact $K$,  for
every $\varepsilon>0$ we have $d_S(f \circ\varphi_{w_n}(z),c) <\varepsilon$
for all     $z\in K$    if   $n\ge n_0$, where  $n_0$   is big enough. Thus,
$f\circ\varphi_{w_n}(K)\subseteq D_S(c,\varepsilon)$ for all $n\ge n_0$. It
follows $f(\bigcup_{n\ge n_0}\varphi_{w_n}(K))\subseteq D_S(c,\varepsilon)$.
In  other    words,     $d_S(f(z),c)<\varepsilon$ if $z\in \bigcup_{n\ge n_0}
\varphi_{w_n}(K)$. Now, let $\{z_k\}\subseteq \bigcup_{n\ge n_0}\varphi_{w_n}
(K)$ be a sequence satisfying $\lim_{k\rightarrow\infty} z_k = e^{i\theta}$.
We will prove that $\lim_{k\rightarrow\infty} f(z_k) =  c$. For every $k\in
\mathbb{N}$ there exist $n_k\in\mathbb{N}$ such that $z_k\in\varphi_{n_k}(K)$.
Since $\lim_{k\rightarrow\infty} z_k=e^{i\theta}$, we have $n_k \rightarrow
\infty$ as $k\rightarrow  \infty$. If $k$ is big enough,   $k\ge k_0$, then
$n_k\ge n_0$      and we have $\{z_{k}:k\ge k_0\}\subseteq \bigcup_{n\ge n_0}
\varphi_n(K)$.  Thus, $d_S(f(z_k),c)<\varepsilon$.  In other words $\lim_{k
\rightarrow\infty} f(z_k)=c$.

{\it (2) implies (1)}. Let us prove the contraposition, that is the negation
of (1) implies the negation of (2). If the  sequence of functions  $\{f_n\}$
does not converge  to $c$ in the local topology of $\mathbb{D}$ to $c$, then
there exist a  compact        set  $K\subset\mathbb{D}$,   a positive number
$\varepsilon_0$, a  subsequence  $\{f_{n_k}\}$,  and a sequence     $\{z_k\}
\subseteq K$ such that $d_S(f_{n_k}(z_k),c)= d_S(f\circ\varphi_{n_k}(z_k),c)
\ge \varepsilon_0$ for all $k\in\mathbb{N}$. Denote $u_k=\varphi_{n_k}(z_k)$.
Then we have $\lim_{k\rightarrow\infty} u_k=e^{i\theta},\ \{u_k\}  \subseteq
\bigcup_{n\in\mathbb{N}}\varphi_n(K)$,   and $d_S(f(u_k),c)\ge\varepsilon_0$.
This   means $C(f,\bigcup_{n\in\mathbb{N}}\varphi_{w_n} (K),e^{i\theta})\not
\equiv \{c\}$, what is negation of (2).
\end{proof}

\begin{definition}\label{DEF3}
A function $f:\mathbb{D}\rightarrow\overline{\mathbb{C}}$ has the   $\Delta_
\gamma-$boundary value  $c\in\overline{\mathbb{C}}$ along the curve $\gamma$
which terminates in  $e^{i\theta}$ if $C(f,\Delta_r\gamma,e^{i\theta})   =
\{c\}$ for all $r\in (0,1)$, i.e.,
\begin{equation*}
\bigcup_{r\in (0,1)}C(f,\Delta_r\gamma,e^{i\theta}) =\{c\}.
\end{equation*}
\end{definition}

\begin{remark}
If   a curve    $\gamma$  lies in  some Stolz  angle with vertex in a point
$e^{i\theta}\in\Gamma$,  then    $\Delta_\gamma-$boundary value is the same
as the ordinary    angular boundary value  of $f$ in $e^{i\theta}$. In this
case  the point $e^{i\theta}$ is the Fatou point for the function $f$.   If
a curve $\gamma\subseteq\mathbb{D}$ is tangent to $\Gamma$ in $e^{i\theta}$
and the  order of contact is $1$, t hen  $\Delta_\gamma-$boundary value for
$f$ is the horo--cycle  boundary value for $f$ in $e^{i\theta}$.
\end{remark}

\begin{theorem}\label{TH2}
Let $f:\mathbb{D}\rightarrow\overline{\mathbb{C}}$ be any function  in the
unit disc and let a simple curve $\gamma\subseteq\mathbb{D}$ terminates in
a point $e^{i\theta}\in\Gamma$.  The        following three conditions are
equivalent:
\begin{enumerate}
\item there exist
$\Delta_\gamma-$boundary     value equal to $c\in\overline{\mathbb{C}}$ in
the point $e^{i\theta}$, i.e.,
\begin{equation*}
\bigcup_{r\in (0,1)}C(f,\Delta_r\gamma,e^{i\theta})=\{c\};
\end{equation*}
\item for any sequence $\{w_n\}\subseteq\gamma$ satisfying $\lim_{n\rightarrow
\infty}w_n=e^{i\theta}$ the sequence $\{f_n=f\circ\varphi_{w_n}\}$ converges
to the constant $c$ in the local topology  of  $\mathbb{D}$;
\item for any curve $\gamma_1\thicksim\gamma$ holds $\lim_{\gamma_1\ni z\rightarrow e^{i\theta}} f(z) = c$.
\end{enumerate}

Moreover, if there exists  $\Delta_\gamma-$boundary  value for $f$, then
it  does not depend   on the choice of a curve  in  the class $[\gamma]$.
\end{theorem}

\begin{proof}
It    follows from Theorem \ref{TH1} that (1)  $\Leftrightarrow$ (2).  It is
clear that  (1)  implies (3). Let  us now prove that (3) implies (1). If (3)
holds then $\lim_{\gamma_1\ni z\rightarrow e^{i\theta}}f(z)=c$; suppose that
(1) is not true. Then for some $r\in (0,1)$ we have $C(f,\Delta_r\gamma,e^{i\theta})
\not\equiv\{c\}$. This means there exist  a sequence $\{z_n\}\subseteq\Delta
_r\gamma,\ \lim_{n\rightarrow\infty}z_n=e^{i\theta}$ such that $\lim_{n\rightarrow
\infty}f(z_n)=a\neq c$ or the previous boundary value does not exist. Points
of the sequence $\{z_n\}$ connect with a curve $\gamma_1$ (in any way)  such
that $\gamma_1\subseteq \Delta_r\gamma$. Now  we have $\lim_{\gamma_1\ni  z
\rightarrow e^{i\theta}} f(z)=c$     or  this limit does  not  exist.  Since
$\gamma_1   \thicksim  \gamma$, from  (3)  it follows that our assumption is
not correct. Thus, for every $r\in (0,1)$ we have $C(f,\Delta_r\gamma,e^{i\theta})
=\{c\}$.
\end{proof}

\begin{definition}\label{DEF4}
Let $f:\mathbb{D}\rightarrow \overline{\mathbb{C}}$ be any function. If for
every  curve $\gamma_1\in[\gamma ]$ we have $\lim_{\gamma_1\ni z\rightarrow
e^{i\theta}} f(z) =  c\in\overline{\mathbb{C}}$, then  we say  that $c$  is
$[\gamma]-$boundary value for $f$, i.e., the boundary    value of $f$ along
the  class $[\gamma]$.
\end{definition}

From  Theorem \ref{TH2} we have

\begin{corollary}\label{CORO2}
Let $f:\mathbb{D}\rightarrow\overline{\mathbb{C}}$ be an arbitrary function
and   let $\gamma$ be a simple curve which terminates in $\Gamma$. Then the
$\Delta_\gamma-$boundary  value  for $f$ exists if and only if there exists
$[\gamma]-$boundary   value for $f$ and they coincides.
\end{corollary}

\begin{remark}
Since  we  have infinity many classes of equivalences in the family of curves
which terminate  in  a point $e^{i\theta}$, we can speak  about infinity many
$\Delta_\gamma-$boundary    values   for   $f$    in the point $e^{i\theta}$.
\end{remark}

\begin{definition}\label{DEF6}
We say that a function   $f:\mathbb{D}\rightarrow\overline{\mathbb{C}}$  is
normal  along a curve $\gamma$ in $D(r)$ where $r\in (0,1)$   if the family
$\{f_w= f\circ\varphi_w:w\in\gamma\}$ is normal in $D(r)$  in  the sense of
Montel.
\end{definition}

\begin{definition}\label{DEF5}
A function $f:\mathbb{D}\rightarrow \overline{\mathbb{C}}$ is normal along a
simple  curve  $\gamma\subseteq\mathbb{D}$ if the family     $\{f_w = f\circ
\varphi_w : w\in\gamma\}$ is normal in the disc $\mathbb{D}$ in the sense of
Montel.
\end{definition}

\begin{corollary}\label{CORO3}
If $f:\mathbb{D}\rightarrow\overline{\mathbb{C}}$ has $\Delta_\gamma-$boundary
value, then  $f$   is a  normal function along the curve $\gamma$.
\end{corollary}

\begin{proof}
If there exists    $\Delta_\gamma-$boundary value of  $f$ equal to $c$,  then
according to the Theorem \ref{TH2} any subsequence of the family $\{f_w =   f
\circ \varphi_w : w\in\gamma\}$ is convergent to the constant function $c$. Thus, this family is normal.
\end{proof}

\begin{remark}
According to Corollary \ref{CORO3} necessary condition for the existence of
$\Delta_\gamma-$boundary value of a  function   $f$ in a point $e^{i\theta}
\in\Gamma$ is normality of $f$ along the curve $\gamma$. However, normality
of the family  $\{f_w=f\circ\varphi_w: w\in\gamma\}$ in $\mathbb{D}$ is not
a sufficient condition. This shows the example of the meromorphic  function
from \cite{NOSHIRO.HOKKAIDO}, which does not have the radial boundary value
in any point of $\Gamma$  and   thus it does not have  the angular boundary
value.   This  also follows from       the example of an   elliptic modular
function     which is a normal analytic  function in   $\mathbb{D}$ but has
the radial boundary value (and thus) only in a countable subset of $\Gamma$.
\end{remark}

\begin{remark}\label{RE4}
Nosiro \cite{NOSHIRO.HOKKAIDO} considered normal meromorphic functions in the
disc $\mathbb{D}$  of the first order: a meromorphic function $f$ in the disc
$\mathbb{ D }$ is normal function of the first order if the family $\{ f\circ
\varphi_w:w\in\mathbb{D} \}$ is normal $\mathbb{D}$ and if any       boundary
function of this family is not a constant.          Nosiro proved that normal
meromorphic   functions of the first order does not poses an angular boundary
value. This result  follows from from our Theorem \ref{TH1}. Our theorem also
shows that a    normal    meromorphic  function of the  first order  does not
have $\Delta_\gamma-$boundary value in any point  of $\Gamma$ for  any  curve
$\gamma$    in disc $\mathbb{D}$ which   terminates in $e^{i\theta}\in\Gamma$.
Since Theorem \ref{TH1}    holds  for  any  function in $\mathbb{D}$  and for
$\Delta_\gamma-$boundary values, our theorem is a generalization of the result
Nosiro.
\end{remark}

\begin{remark}\label{RE5}
Bagemihl    and Seidel \cite{BAGEMIHL.SEIDEL.AASF}  constructed an analytic
function (see Example 2 there) which proves that in Theorem 1 in  \cite{BAGEMIHL.SEIDEL.AASF}
the condition of  normality of  the  function in $\mathbb{D}$  in order  to
has a  boundary value cannot be removed.   Corollary \ref{CORO3}       also
shows   that.  Corollary \ref{CORO3} is a  generalization of the  result of
Bagemihl    and Seidel for any function  in $\mathbb{D}$  and   curvilinear
boundary behavior.
\end{remark}

\begin{remark}\label{RE6}
Theorem \ref{TH2} and Corollary \ref{CORO3} also show that for the existence
of angular  boundary values of functions in a point  of $\Gamma$ (meromorphic,
analytic, harmonic,  etc.)  one need not assume their normality in the  disc
$\mathbb{D}$,  i.e., it is  not necessary    to   assume that a function  is
normal with respect to the   M\"{o}bius group of  conformal automorphisms of
$\mathbb{D}$ (see \cite{LEHTO.VIRTANEN.ACTA}), or one need not to assume the
condition of normality with respect the  hyperbolic  or  parabolic subgroups
or semigroups of the M\"{o}bius group (see      \cite{GAVRILOV.BURKOVA.DOKL},
\cite{LOVSHINA.DOKL} and \cite{PAVICEVIC.SUSIC.DOKL}). It is enough to assume
the condition of normality  of the family $\{f_w=f\circ \varphi_w:w\in\gamma
\}$     in $\mathbb{D}$, i.e., that $f$ is normal function along the   curve
$\gamma$ which is not tangent to $\Gamma$.
\end{remark}

The  following theorem is generalization of Theorem 1 in Lehto and Virtanen
work \cite{LEHTO.VIRTANEN.ACTA}     for  an arbitrary function in the  disc
$\mathbb{D}$ and for any  $\Delta_\gamma-$boundary limit.

\begin{theorem}\label{TH3}
Let $f:\mathbb{D}\rightarrow\overline{\mathbb{C}}$ be any function in the
disc  $\mathbb{D}$ and let $\gamma\subseteq\mathbb{D}$  be  a curve which
terminates in $e^{i\theta}\in\Gamma$.     Suppose      $\lim_{\gamma\ni z
\rightarrow e^{i\theta}}f(z) = c\in\overline{\mathbb{C}}$   and    assume
that $f$ does not have $\Delta_\gamma-$boundary limit.     Then for every
$\varepsilon>0$ there exist two curves $\gamma_1,\, \gamma_2\in [\gamma]$
such that the (pseudo--)hyperbolic   distance     between $\gamma_2$  and
$\gamma_1$ is  less then $\varepsilon$ and such that along $\gamma_1$ the
function  $f$ has the asymptotic boundary value $c$, and along $\gamma_2$
does not.
\end{theorem}

\begin{proof}
In the proof we will use the hyperbolic metric  $d_h$.

Since  $\lim_{\gamma\ni w\rightarrow e^{i\theta}}f(w)=c$ and since  $f$ does
not have $\Delta_\gamma-$boundary limit,      from the first part of Theorem
\ref{TH1} it follows that $C(f,\Delta_{r_0}\gamma, e^{i\theta})\not   \equiv
\{c\}$ for some $r_0\in (0,\infty)$. This means that the  set $\{r:C(f,\Delta_{r}
\gamma,e^{i\theta})\equiv \{c\}:0\le r<\infty\}$    is bounded from above by
$r_0$.    Thus, there  exist
\begin{equation*}
r_1=\sup \{r:C(f,\Delta_{r}\gamma,e^{i\theta})\equiv \{c\}:0\le r<\infty\}.
\end{equation*}
Let $\varepsilon$ be any positive number. We will consider the following two
cases.

{\it The case}   $r_1=0$.    Then $\gamma=\gamma_1$ and according to  Theorem
\ref{TH2} in the curvilinear angle $\Delta_{\frac \varepsilon 2}\gamma$ there
exist   a curve $\gamma_2$ along which function $f$ does not poses asymptotic
value $c$.  In view of Definition \ref{DEF1} the hyperbolic distance  between
curves $\gamma_2$     and $\gamma_1$ is less then $\varepsilon$.

{\it The case} $r_1\in (0,\infty)$.  We will denote $r_1$ with $r$. Denote by
$\gamma_r^+$   the part of the boundary of curvilinear angle $\Delta_r\gamma$
which is above     the curve $\gamma$, and with $\gamma_r^-$ the part of the
boundary of $\Delta_r\gamma$     which is below the curve $\gamma$. For every
$w'\in\gamma_r^+$  there exist  $w\in\gamma$  such that $d_h(w,w')=r$. On the
circle  (in the metric $d_h$)  which contains $w$ and $w'$ and  is orthogonal
on  $\Gamma$,  take the points $u$  which are between $w$ and $w'$  and which
satisfy  $d_h(w,u)=r-\frac \varepsilon 4$.     Points $u$   make the part  of
boundary of curvilinear  angle $\Delta_{\frac \varepsilon 4}\gamma$ which  is
on the same side of the curve $\gamma$ as $\gamma_r^+$.      That part of the
boundary produces a curve which will be denoted by $\lambda'$. Since $\lambda'
\subseteq\Delta_r\gamma$,  along $\lambda'$ there exists asymptotic  value of
$f$ equals to $c$. Consider now the curvilinear angle $\Delta_{\frac\epsilon2}
\lambda'$.   Denote with $\Delta^+_{\frac\epsilon 2}\lambda'$ the sub--domain
of curvilinear angle $\Delta_{\frac \epsilon 2}\lambda'$ which  is bounded by
the curves $\gamma_r^+$ and ${\lambda'}^+_{\frac\varepsilon 2}$. Let  $r_2\in
(r_1,\infty)$.  With $(\gamma^+_{r_1},\gamma^+_{r_2}]$  denote the     domain
$\Delta_{r_2}^+\gamma\setminus\Delta_{r_1}^+\gamma$, and with $(\gamma^-_{r_1},
\gamma^-_{r_2}]$ the domain $\Delta_{r_2}^-\gamma\setminus\Delta_{r_1}^-\gamma$.
In view of the preceding notations, it is not hard to see that $\Delta^+_{\frac\varepsilon2}
\lambda'\subseteq(\gamma^+_r,\gamma^+_{r+\frac \varepsilon 4}]$, what is
obvious from  the geometric interpretation  of these sets. Similarly,
$(\gamma^+_r,\gamma^+_{r+\frac \varepsilon 4}]\subseteq\Delta^+_{\frac \varepsilon 2}\lambda'$.
Namely, if  $z\in (\gamma^+_r,\gamma^+_{r+\frac\varepsilon 4}]$, then there
exists a disc $D_h(w,r+\frac\epsilon 4),\ w\in\gamma$ such that $z\in D_h(w,r+\frac \varepsilon4)$.
Let  $u$  be a point which is in the intersection of the curve $\lambda'$ and
the hyperbolic  half--radius  of the  disc  $D_h(w,r+\frac \varepsilon 4)$
which contains $z$. Then  $z\in D_h(u,\frac \varepsilon 2)$. Thus   $z\in
\Delta^+_{\frac \varepsilon 2}\lambda'$, what implies
$(\gamma^+_r,\gamma^+_{r+\frac \varepsilon 4}]\subseteq\Delta^+_{\frac \varepsilon 2}\lambda'$.
It follows $\Delta^+_{\frac \varepsilon 2}\lambda'\subseteq (\gamma^+_r,\gamma^+_{r+\frac \varepsilon 4}]$.

All we done for $\gamma^+_r$ may be done also for $\gamma^-_r$. The resulting
curve will be denoted by $\lambda''$.

In some of the sub--domains $(\gamma^+_r,\gamma^+_{r+\frac \varepsilon 4}]$
and $(\gamma^-_r,\gamma^-_{r+\frac \varepsilon 4}]$  there exist a sequence
of points along which the function $f$ does not have the boundary value, or
if there    exist,     then it is not equals to $c$. This sequence   may be
connected by a curve   $\gamma_2$ which lies in the same sub--domain as the
sequence. Now we may take  $\gamma_1=\lambda'$ or $\gamma_1=\lambda''$ what
depends     on which sub--domain contains the curve $\gamma_2$. It is clear
that the hyperbolic distance  between    the curves          $\gamma_2$ and
$\gamma_1$ is  less then $\varepsilon$.
\end{proof}

\begin{theorem}\label{TH4}
Let      $f:\mathbb{D}\rightarrow\overline{\mathbb{C}}$ be any function in
$\mathbb{D}$ and  let a curve $\gamma\subseteq \mathbb{D}$ terminates in a
point $e^{i\theta}\in\Gamma$. If $C(f,\Delta_r\gamma,e^{i\theta}) = \{c\}$
for some $r\in (0,\infty)$  and  if         the function $f$ does not have
$\Delta_\gamma-$boundary value in the point $e^{i\theta}$, then there exist
$\gamma_1\in[\gamma]$ such that  for every   $\varepsilon>0$ there exist a
curve $\gamma_\varepsilon$  at the  distance less then $\varepsilon$  from
the curve $\gamma_1$,  such   that along  one curve $f$ has the asymptotic
boundary value $c$, and along  the other does not.
\end{theorem}

\begin{proof}
That this theorem holds one may see from the proof of Theorem \ref{TH3}.
One may also chose $\gamma_1=\lambda'$ and $\gamma_\varepsilon=\lambda''$.
\end{proof}

\section{Normality of meromorphic functions along simple curves}
With $N$ we denote the  class      of all normal meromorphic functions in the
disk $\mathbb{D}$    (for                properties of this class we refer to
\cite{LEHTO.VIRTANEN.ACTA},         \cite{LOHWATER},       \cite{MONTEL.BOOK},
\cite{SCHIFF.BOOK}, \cite{BAGEMIHL.SEIDEL.AASF},   \cite{GAVRILOV.SBORNIK.67},
\cite{GAUTHIER.NAGOYA}, \cite{LAPPAN.MATH.Z}, \cite{VAISALA.AASF},
\cite{LAVERDE}).    For    a meromorphic function    $f:\mathbb{D}\rightarrow
\overline{\mathbb{C}}$ we denote with
\begin{equation*}
f^{\sharp}(z) = \frac{|f'(z)|}{ 1+|f(z)|^2},\quad z\in\mathbb{D}
\end{equation*}
the spherical derivate for $f$. The function $f^\sharp:\mathbb{D}\rightarrow
\mathbb{R}$ is continuous in $\mathbb{D}$. The spherical derivate $f^\sharp$
may be used  to  define the spherical distance     between      points in the
target domain of   $f:\mathbb{D}\rightarrow\overline{\mathbb{C}}$.  Ostrowski
\cite{OSTROWSKI.MATH.Z}  was the first who used the spherical distance, i.e.,
the  spherical derivate in the   consideration    related   with meromorphic
functions. Lehto and Virtanen \cite{LEHTO.VIRTANEN.ACTA}   used the   Marty
criterium (see \cite{HAYMAN.BOOK}) in order to derive   that  a  meromorphic
function belongs to the class $N$ if and only if
\begin{equation*}
\sup_{z\in\mathbb{D}} (1-|z|^2) f^\sharp (z)<\infty.
\end{equation*}

By   using   the Marty  criterium for normality of a family of meromorphic
functions  it is a routine to prove the following

\begin{theorem}\label{TH5}
Let a curve $\gamma\subseteq\mathbb{D}$ has one endpoint in $e^{i\theta}\in
\Gamma$ and let $f:\mathbb{D}\rightarrow\overline{\mathbb{C}}$         be a
meromorphic function.   For every $r\in (0,1)$ the following conditions are
equivalent:
\begin{enumerate}
\item $f$ is normal in $D (r)$  along the   curve $\gamma$;
\item $\sup_{z\in\Delta_r\gamma  } (1-|z|^2) f^\sharp (z)<\infty.$
\end{enumerate}
\end{theorem}

From Theorem \ref{TH5} we have

\begin{theorem}[see Theorem 1 in \cite{SUSIC.PAVICEVIC.BALCANICA}]\label{TH6}
Let     a   curve $\gamma\subseteq\mathbb{D}$ terminates in $e^{i\theta}\in\Gamma$
and let $f:\mathbb{D} \rightarrow\overline{ \mathbb{C}}$ be a meromorphic function.  The following
two  conditions are equivalent:
\begin{enumerate}
\item  $f$ is  normal  in $\mathbb{D}$ along $\gamma$;
\item  $\sup_{z\in\Delta_r\gamma} (1-|z|^2) f^\sharp (z)<\infty$      for all
$r\in (0,1)$.
\end{enumerate}
\end{theorem}

\begin{theorem}\label{TH7}
A   meromorphic  function  $f : \mathbb{D}\rightarrow\overline{\mathbb{C}}$ is
normal along a  curve $\gamma$ if and only if it is normal along any curve
$\gamma_1\in [\gamma ]$.
\end{theorem}

\begin{proof}
This theorem follows from Theorem \ref{TH6} and Lemma \ref{LE2}.
\end{proof}

Theorem \ref{TH7} gives an opportunity to introduce a notation of normality
of meromorphic functions  along classes  of simple curves.

\begin{definition}\label{DEF7}
A meromorphic function $f:\mathbb{D}\rightarrow\overline{\mathbb{C}}$ is normal
along  a class   $[\gamma]$   if it is normal along any curve  $\gamma_1\in
[\gamma]$.
\end{definition}

\begin{remark}
From Theorems \ref{TH6} and \ref{TH7}  it follows that the class of normal
meromorphic functions in the  disc $\mathbb{D}$ investigated  by Lehto and
Virtanen \cite{LEHTO.VIRTANEN.ACTA},  i.e., functions from  $N$, is normal
along      any simple curve $\gamma$ in the disc $\mathbb{D}$, i.e., it is
normal along  any class $[\gamma]$  in  the disc $\mathbb{D}$.  If a curve
$\gamma$  lies  in   some Stolz     angle of $\mathbb{D}$ with vertex   at
$e^{i\theta}$, then the      notation of normality along that curve, i.e.,
along the class $[\gamma]$ of a function $f :\mathbb{D}\rightarrow\overline{\mathbb{C}}$
is equivalent with normality along   the hyperbolic semigroup of all M\"{o}bius
transformations of $\mathbb{D}$ with an attractive point   $e^{i\theta}$
($\pm e^{i\theta}$ are attractors of elements of the semigroup).   If  a
curve belongs to the sub--domain of  the disc $\mathbb{D}$    which is bounded
by two horocycles which contain $e^{i\theta}$, then   the normality of meromorphic
functions along the curve $\gamma$,  i.e., along the   class $[\gamma]$ is
the same as  the normality along the parabolic semigroup of all M\"{o}bius
transformations of the disc     $\mathbb{ D}$   with only  one  attractive
point $e^{i\theta}$. Normality and boundary behavior of meromorphic functions
along the hyperbolic   and parabolic semigroup and hyperbolic and parabolic
subgroup are considered in \cite{GAVRILOV.BURKOVA.DOKL}, \cite{LOVSHINA.DOKL}, \cite{PAVICEVIC.SUSIC.DOKL},
\cite{GAVRILOV.SBORNIK.67} and \cite{PAVICEVIC.NEW.ZELAND}.
\end{remark}

In  \cite{GAVRILOV.SBORNIK.67}  and     \cite{GAVRILOV.SBORNIK.71}  Gavrilov
considered the   normality and boundary behavior of     meromorphic function
using the notation of $P-$sequences.

\begin{definition}[see \cite{GAVRILOV.SBORNIK.67}]\label{DEF8}  A  sequence
$\{z_n \}\subseteq \mathbb{D},\, \lim_{n\rightarrow \infty}|z_n|=1$    is a
$P-$sequence   for a meromorphic function $f:\mathbb{D}\rightarrow\overline
{\mathbb{C}}$    if for  any subsequence $\{z_{n_k}\}$ and     $\varepsilon
\in(0,1)$ the function $f$  achieves   in the set $\bigcup _{k\in\mathbb{N}}
D_h(z_{n_k}, \varepsilon)$                infinity many times all values in
$\overline{\mathbb{C}}$ except  possibly two.
\end{definition}

From  the definition it follows that any subsequence of $P-$sequence is also
a $P-$sequence.

Gavrilov     (see   Theorem 3 in \cite{GAVRILOV.SBORNIK.67})   showed that a
meromorphic function $f$ is normal  in the disc $\mathbb{D}$, i.e., $f\in N$
is and only  if    $f$   in  $\mathbb{D}$ does not  have $P-$sequences.

\begin{theorem}[see Theorem    1 in \cite{GAVRILOV.SBORNIK.71}]\label{TH8} If
for a sequence $\{z_n\}\subseteq\mathbb{D},\ \lim_{n\rightarrow\infty}  |z_n|
=1$ and a meromorphic function $f:\mathbb{D}\rightarrow\overline{\mathbb{C}}$
hold
\begin{equation*}
\lim_{n\rightarrow\infty} (1 - |z_n|^2) f^\sharp(z) =  \infty,
\end{equation*}
then  $\{z_n\}$ is  a $P-$sequence for $f$.
\end{theorem}

The example of the meromorphic function
\begin{equation*}
f(z ) =   \exp\left \{ - \exp \left\{ \frac 1{1-z}\right\}\right\}
\end{equation*}
from \cite{GAVRILOV.SBORNIK.67} shows that the reverse       implication in
Theorem \ref{TH8} does not hold.

\begin{theorem}[see Theorem 3 in \cite{GAVRILOV.SBORNIK.71}]\label{TH9}   A
sequence  $\{z_n\}\subseteq \mathbb{D},\, \lim_{n\rightarrow\infty}|z_n|=1$
is a  $P-$sequence     for a meromorphic  function $f:\mathbb{D}\rightarrow
\overline{\mathbb{C}}$  if and only   if there exist a sequence of positive
numbers  $\{r_n\},\, \lim_{n\rightarrow \infty}r_n=0$ such that
\begin{equation*}
\lim_{n\rightarrow\infty} \left\{\sup_{z\in D_h(z_n,r_n)} (1-|z|^2)f^\sharp (z)\right\} =\infty.
\end{equation*}
\end{theorem}

\begin{theorem}[see Theorem 2 and  Theorem 5 in \cite{GAVRILOV.SBORNIK.67}]\label{TH10}
For  a meromorphic function $f:\mathbb{D}\rightarrow\overline{\mathbb{C}}$
let $\{z_n\}\subseteq\mathbb{D},\ \lim_{n\rightarrow\infty}|z_n| = 1$ be a
sequence which satisfies $\lim_{n\rightarrow\infty}f(z_n)=\alpha\in\overline {\mathbb{C}}$.
Let $\{z_n'\}$ be a new sequence such that along this one the function $f$
does not poses  a limit  $\alpha$ and  $\lim_{n\rightarrow\infty} d_h(z_n,z_n')
=0$.   Then each of   $\{z_n\}$ and $\{z_n'\}$ are $P-$sequences.
\end{theorem}

\begin{theorem}\label{TH11}
Let a curve $\gamma\subseteq \mathbb{D}$ terminates in $e^{i\theta}\in\Gamma$.
A     meromorphic function  $f:\mathbb{D}\rightarrow\overline{\mathbb{C}}$ is
normal  along a  curve   $\gamma$ in $D(r)$,  where $r\in (0,1)$ is fixed, if
and       only if the function $f$ does not have a $P-$sequenece in $\Delta_r
\gamma$.
\end{theorem}

One  can give a proof of Theorem \ref{TH11}   in the similar way as the proof
of the  following

\begin{theorem}[see Theorem 3 in \cite{SUSIC.PAVICEVIC.BALCANICA}]\label{TH12}
A     meromorphic function  $f:\mathbb{D}\rightarrow\overline{\mathbb{C}}$ is
normal along  a simple curve $\gamma\subseteq\mathbb{D}$  which terminates in
$e^{i\theta}\in\Gamma$ if and only if  for all  $r\in (0,1)$ the function $f$
does  not have a $P-$sequence in $\Delta _r\gamma$.
\end{theorem}

From Theorem \ref{TH11} we immediately deduce

\begin{theorem}\label{TH13}
Let $f:\mathbb{D}\rightarrow\overline{\mathbb{C}}$  be a meromorphic function
such that there exits a domain $O$ which contains $e^{i\theta}\in\Gamma$ such
that for all $r\in (0,1)$ the function $f$      is bounded in $O\cap \Delta_r
\gamma$. Then  $f$ is    normal function along $\gamma$, i.e., $f$  is normal
along $[\gamma]$.
\end{theorem}

Altogether,  from  the results of this section we obtain

\begin{proposition}
For    a meromorphic function  $f:\mathbb{D}\rightarrow\overline{\mathbb{C}}$
and a curve $\gamma$  the following conditions are equivalent:
\begin{enumerate}
\item  $f$ is normal along $\gamma$;
\item $f$  is normal along $[\gamma]$;
\item for all $r\in (0,1)$  holds
\begin{equation*}
\sup_{z\in\Delta_r\gamma} (1-|z|^2)f^\sharp (z)<\infty;
\end{equation*}
\item $f$ does not have $P-$sequences in $\Delta_r\gamma,\, r\in (0,1)$.
\end{enumerate}
\end{proposition}

This characterization could be stated also  on the normality of $f$ along
$\gamma$ on $D(r)$ for each $r\in (0,1)$.

$P-$sequences,   as shows Theorem \ref{TH3} in \cite{GAVRILOV.SBORNIK.67},
characterize boundary behavior of meromorphic functions in the unit disc
(for example see  also \cite{SUSIC.PAVICEVIC.BALCANICA}, \cite{GAVRILOV.BURKOVA.DOKL},
\cite{LOVSHINA.DOKL},  \cite{PAVICEVIC.SUSIC.DOKL}, \cite{GAVRILOV.SBORNIK.67},
\cite{PAVICEVIC.NEW.ZELAND}, \cite{GAVRILOV.SBORNIK.71}, \cite{PAVICEVIC.MOSCOW.BULL},
\cite{GAUTHIER.NAGOYA}, \cite{GAUTHIER.CANADIAN.J.MATH}).   We  will use
them in the  seventh section for the  construction of        meromorphic
functions showing that  the results from this paper   cannot be improved.

\section{Curvilinear boundary values of meromorphic functions}

\begin{theorem}\label{TH14}
Let $f:\mathbb{D}\rightarrow\overline{\mathbb{C}}$ be a meromorphic function,
let a curve $\gamma\subseteq\mathbb{D}$  terminates in $e^{i\theta}\in\Gamma$,
and let $c\in\overline{\mathbb{C}}$. The following conditions are equivalent:
\begin{enumerate}
\item
$f$ is normal along $\gamma$ and $\lim_{\gamma\ni  z\rightarrow e^{i\theta}}f(z)=c$;
\item $c$ is $\Delta_\gamma-$boundary value of $f$.
\end{enumerate}
\end{theorem}

\begin{proof} {\it (1) implies (2)}:       Since $f$ is normal along $\gamma$,
for any sequence $\{w_n\}\subseteq\gamma$ which satisfies $\lim_{n\rightarrow
\infty} w_n = e^{i\theta}$ there exist a  subsequence $\{w_{n_k}\}$ such that
the $\{f_{n_k}(z)= f\circ \varphi_{{n_k}}(z)\}$,     where $\varphi_{{n_k}} =
\varphi_{w_ {n_k}}$,  is convergent to a    (meromorphic) function $\psi$ (in
the local topology of $\mathbb{D}$). Let $r_1\in (0,1)$             and $K  =
\overline{D}_{ph}(0,r_1) = \overline{D}(r_1)  =  \{z:|z|\le r_1\}$. For $r\in
(0,r_1)$  let us consider $\gamma\cap \overline{D}_{ph}(w_{n_k},r_1)\setminus
D_{ph}(w_{n_k},r)$. This set consists of two curves; let $\gamma_k$  be   one
of them and denote  $\Gamma_k= \varphi_{n_k}^{-1}(\gamma_k)$.  For all  $m\in
\mathbb{N}$         let $\{z_k^m\in\Gamma_k\}$ be any sequence that satisfies
$\lim_{k  \rightarrow\infty}z_k^m=z_0^m\in \overline{D}_{ph}(0,r_1)$. We will
show  that $\psi(z_0^m)=c$ for all $m\in\mathbb{N}$. Namely, for all    $m\in
\mathbb{N}$ we have
\begin{equation*}
d_S(\psi(z_0^m),c)\le
d_S(\psi(z_0^m),\psi(z_k^m))+d_S(\psi(z_k^m),f_{n_k}(z_k^m))+d_S(f_{n_k}(z_k^m),c)
\end{equation*}
Let     $\varepsilon>0$     be any number. Because   of continuity  of $\psi$
we have $d_S(\psi(z_0^m) ,\psi(z_k^m))<\frac\varepsilon 3$,  if $k$    is big
enough. Since the sequence  $\{f_{n_k}\}$   is  convergent  to  $\psi$ in the
local   topology      of $\mathbb{D}$, we have $d_S(\psi(z),f_{n_k}(z))<\frac
\varepsilon3$ for all     $z\in \overline{D}_{ph}(0,r_1)$   and if $k$ is big
enough. Since $z_k^m\in\overline{D}_{ph}(0,r)$,      we have $d_S(\psi(z_k^m),
f_{n_k}(z_k^m))< \frac \varepsilon 3$.  Since $z_k^m\in\Gamma_k$,  it follows
$\varphi_{n_k}(z_k^m) = w_k^m\in\gamma_k\subseteq\gamma$  and $\lim _{k\rightarrow\infty} w_k^m  =
e^{i\theta}$.     Since $c$ is an asymptotic boundary  value of $f$ and since
$\lim_{k\rightarrow\infty}f_{n_k}(w_k^m)=\lim_{k\rightarrow\infty}     f\circ
\varphi_{n_k}(z_k^m)    =  \lim_{k\rightarrow \infty } f_{n_k}(z_k^m )=c$ for
big enough $k$, we have $d_S(f_{n_k}(z_k^m),c)<\frac \varepsilon3$    for all
$m\in \mathbb{N}$. From the preceding inequalities it follows $d_S(\psi(z_0^m),c)
< \varepsilon$ for all $m\in\mathbb{N}$.       Since    $\varepsilon$  is any
positive number, it must be $\psi(z_0^m)=c$ for all   $m\in\mathbb{N}$. Since
the sequence $\{z_0^m\}$ lies in $\overline{D}_{ph}(0,r_1)$ and since in this
set it has an accumulation point, from the uniqueness theorem,  we have $\psi
\equiv c$.

Thus,   it   is  showed     that     any    convergent sequence (in the local
topology of $\mathbb{D}$)  of the family   $\mathcal {F}_{f,\gamma}= \{f\circ
\varphi_w:w\in\gamma\}$ converges to $c$. We will show now that any  sequence
of   the family $\mathcal{F}_{f,\gamma}$ converges to the constant $c$ in the
local topology. Assume contrary, let that  there exist a    sequence $\{f_n\}
\subseteq \mathcal{F}_{f,\gamma}$,  which   is    not convergent in the local
topology to the constant $c$.  There exist $\varepsilon > 0 $ such that   for
any   $k\in\mathbb{N}$ there exist        $n_k\in\mathbb{N}$ and  $z_{n_k}\in
\overline{D}_{ph}(0,r)$ such that     $d_S(f_{n_k}(z_{n_k}),c)\ge\varepsilon$.
Since        the           family $\mathcal{F}_{f,\gamma}$ is normal,     the
sequence $\{f_{n_k}\}$ has a subsequence $f_{n_{k_l}}$ which is   convergent;
according to the preceding, it converges  to the constant $c$,  what       is
contrary to the assumption  $d_S(f_{n_k}(z_{n_k}) ,c ) \ge \varepsilon$. This
contradiction shows that any sequence in $\mathcal {F}_{f,\gamma}$  converges
in the local topology  of  $\mathbb{D}$  to the constant $c$.   From this and
from Theorem \ref{TH1} we have $C(f,\Delta_r\gamma,e^{i\theta})= \{c\}$   for
all $r\in (0,1)$, that is the function $f$ in the point $e^{i\theta}$     has
$\Delta_\gamma-$boundary    value    along the curve $\gamma$ equal to $c$.

{\it (2) implies (1)}: Form Theorem \ref{TH1} and condition (2) we have $C(f,
\Delta_r \gamma, e^{i\theta})= \{c\}$ for all $r\in (0,1)$. It follows   that
any sequence $\{f_n\}\subseteq\mathcal{F}_{f,\gamma}$ converges        to the
constant $c$. We infer that   $\mathcal {F}_{f,\gamma}$ is normal family   in
$\mathbb{D}$. Regarding  Definition \ref{DEF4}  this means that $f$ is normal
along the curve $\gamma$. From the condition $C(f,\Delta_r\gamma,e^{i\theta})
=\{c\}$ evidently follows $\lim_{\gamma\ni z\rightarrow e^{i\theta}} f(z)=c$.
\end{proof}

\begin{remark}
Seidel and Walsh proved (see Theorem 4, p. 199 in \cite{SEIDEL.WALSH.TAMS}):
Let $f$ be an  analytic function in $\mathbb{D}$ which omits at least    two
values. Let $\gamma_1$ and $\gamma_2$ be simple curves in $\mathbb{D}$ which
terminate       in   $1$      with  finite        Fr\'{e}shet   distance. If
$\lim_{\gamma_1\ni z \rightarrow 1}f(z)  =  c\in\overline{\mathbb{C}}$, then
also $\lim_{\gamma_2\ni z\rightarrow 1}f(z)=c$. This statement remains valid
if we  assume    that    $f$     is normal meromorphic function  in the disc
$\mathbb{D}$ (see Theorem 2.12, p. 131, in \cite{LOHWATER}).     Our Theorem
\ref{TH14} shows    that   the  theorem of Seidel and   Walsh  holds for all
curves in the class $[\gamma]$.        Lemma  \ref{LE4}   shows that Theorem
\ref{TH14} is  a generalization of the previous results.
\end{remark}

In the similar manner as Theorem \ref{TH14} one can prove the      following

\begin{theorem}\label{TH15}
For a  meromorphic function $f:\mathbb{D}\rightarrow\overline{\mathbb{C}},\,
r\in (0,1)$ and     $c\in\overline{\mathbb{C}}$ the following conditions are
equivalent:
\begin{enumerate}
\item $f$ is normal in  ${D}(r)$ along a curve $\gamma$ and $\lim_{\gamma\ni
z\rightarrow e^{i\theta }} f(z)=c$;
\item $C(f,\Delta_{r_1}\gamma,e^{i\theta}) = c$ for all $r_1\in (0,r)$.
\end{enumerate}
\end{theorem}

Form Theorems \ref{TH6}, \ref{TH14} and \ref{TH15} we conclude

\begin{theorem}\label{TH16}
Assume  a meromorphic function $f:\mathbb{D}\rightarrow\overline{\mathbb{C}}$
has  an   asymptotic boundary value along a curve $\gamma$. In order that $f$
has $\Delta_r\gamma-$boundary value along the curve $\gamma$  it is necessary
and  sufficient   that
\begin{equation*}
\sup_{z\in \Delta_r\gamma} (1-|z|^2)f^\sharp (z)<\infty
\end{equation*}
for  all  $r\in (0,1)$, what is equivalent to the condition  that   $f$ does
not contain   $P-$sequences in $\Delta_r\gamma$.
\end{theorem}

\begin{remark}
Results in Theorems \ref{TH14}, \ref{TH15} and \ref{TH16} are generalization
of Theorems 2, 2', 4 and 5 of Lehto  and Virtanen \cite{LEHTO.VIRTANEN.ACTA}.
They     considered the normal meromorphic functions in $\mathbb{D}$ and its
angular boundary  values. We consider the  normal meromorphic      functions
along  a curve $\gamma$ and  $\Delta_\gamma-$boundary  values.   Particulary,
if $\gamma$ is not tangent to $\Gamma$, from Theorem \ref{TH14} and \ref{TH16}
we derive  Theorems 2, 2', 4  and   5 in  \cite{LEHTO.VIRTANEN.ACTA}.
\end{remark}

\begin{remark}\label{RE7}
Theorem   \ref{TH2}, Theorem \ref{TH14} and Theorem \ref{TH15}    show  that
each of the following two  conditions:
\begin{enumerate}
\item $f$ is normal meromorphic function along a curve $\gamma$;
\item a meromorphic function $f$ has an asymptotic  boundary value along the
curve $\gamma$,
\end{enumerate}
is   a   necessary condition   for the existence of $\Delta_\gamma-$boundary
value  ($[\gamma]-$boundary value) of $f$.

Taken together,      conditions (1) and (2)  are  necessary and    sufficient
for the    existence  of $\Delta_\gamma-$boundary  value ($[\gamma]-$boundary
value)  of  $f$.
\end{remark}

\section{Examples}
The  following   examples   we   will      construct in the similar way as in
\cite{GAVRILOV.BURKOVA.DOKL}.

\begin{example}\label{EX1}
Let   $\gamma$        be a curve terminating in $e^{i\theta}\in\Gamma$. With
$\gamma_r^+$  and $\gamma_r^-$ we denote the    parts of the boundary of the
curvilinear angle $\Delta_r\gamma,\, r\in (0,1)$  as in the proof of Theorem
\ref{TH3}. Let $\{z_k\},\, \lim_{k\rightarrow \infty}z_k = e^{i\theta}$ be a
sequence       of points such  that $z_{2m}\in\gamma^+_{r_m},\,  z_{2m-1}\in
\gamma^-_{r_m}$, where $r_m\uparrow 1$. Moreover, let $\{\varepsilon_k\}$ be
a sequence of numbers  which satisfies:
\begin{enumerate}
\item $0<\varepsilon_{k+1}<\varepsilon_k$ for all $k\in\mathbb{N}$;
\item  $\lim_{k\rightarrow\infty}\varepsilon_k=0$;
\item $D_i\cap D_j$ is the empty set for all $i\ne j$, where $D_k =\{z:|z-z_k|
<\varepsilon_k\}$;
\item $\lim_{k\rightarrow \infty} \sup_{z\in D_k}d_h(z,z_k) = 0$;
\item $\sum_{k=1}^\infty\varepsilon_k<\infty$.
\end{enumerate}
Let          $a_k=\varepsilon_k^2$  for all  $k\in\mathbb{N}$ and $f_0(z) =
\sum_{k=1}^\infty a_k (z-z_k)^{-1}$. For fixed $n\in\mathbb{N}$,    we have
\begin{equation*}
\left|\sum_{k\ne n} a_k(z-z_k)^{-1}\right|\le\sum_{k=1}^\infty \varepsilon_k^2
<\infty.
\end{equation*}
It follows that $f_0$ is a meromorphic function   in  the disc $\mathbb{D}$
with poles at $z_k,\,  k\in\mathbb{N}$.  Since   $|f_0(z_k+\varepsilon_k)|<
\infty$ and $\lim_{k\rightarrow\infty} d_h (z_k,z_k+\varepsilon_k)=0$, from
Theorem  \ref{TH10} it follows that         the  sequence    $\{z_k\}$ is a
$P-$sequence for $f_0$.     Since for all $z',\, z'' \in\mathbb{D}\setminus
\bigcup_{k\in\mathbb{N}} D_k$ holds
\begin{equation*}
|f_0(z')-f_0(z'')|\le |z'-z''|  \sum _{k=1}^\infty \varepsilon_k
\end{equation*}
and since any of sets $\Delta_r\gamma,\, r\in (0,1)$ contains     a  finite
number of   points  from $\{z_k\}$,   it  follows  that
\begin{equation*}
\limsup _{\Delta_r \gamma\ni z \rightarrow e^{i\theta }} |f_0(z)| =  c_{f_0}(r)<\infty,\quad r\in (0,1).
\end{equation*}
Hence, for all $r\in (0,1)$, the function $f_0$  is  bounded   in $O_r\cup
\Delta_r\gamma$,  where $O_r = \{z: |z-e^{i\theta}|<1-r\}$.   From Theorem
\ref{TH13}, it follows that $f_0$ is normal in  $O_r\cup   \Delta_r\gamma$.
Now,  from Theorem \ref{TH11}  we obtain  that   $f_0$  is  normal   along
the curve $\gamma$        (see \cite{SCHIFF.BOOK}, p. 35, Montel's theorem).

The   way we constructed the function $f_0$ shows that any set $\mathcal{A}$
which contains all sets $O_r\cap \Delta_r\gamma,\, r\in (0,1)$,     and any
sequence of  points $\{z_k\}$ contains a $P-$sequence of the function $f_0$.
It    is possible to show that there exist vicinities $O_r,\, r\in(0,1)$ of
$e^{i\theta}$ such that
\begin{equation*}
\left(\bigcup_{r\in (0,1)}O_r\cap  \Delta_r\gamma\right)\cup\{z_k: k\in\mathbb{N}\}
\varsubsetneq  D\cup O_{e^{i\theta}},
\end{equation*}
where    $O_{e^{i\theta}}$    is any vicinity   of the  point $e^{i\theta}$.
\end{example}

The     preceding    facts  can be   illustrated well  if we take  for a set
$\mathcal{A}$    a    horo--cycle  which is tangent to $\Gamma$ in the point
$e^{i\theta}$ and for  the curve $\gamma$ a radius od   $\mathbb{D}$    with
one endpoint in       $e^{i\theta}$. Then the domain $\Delta_r\gamma$ is the
sub--domain of the disc  $\mathbb{D}$ which  is bounded by two hyper--cycles
which contain $\pm e^{i\theta}$ (see Figure 1).
\begin{figure}[htp]\label{FIG1}
\centering
\includegraphics{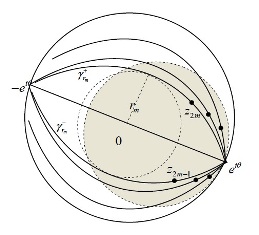}
\caption{}
\end{figure}

Thus, the function $f_0$ shows  that in a general case  does not exist a set
which contains all  sets $\Delta_r\gamma,\, r\in(0,1)$ such that  along that
set  the function does not poses $P-$sequences.

\begin{example}\label{EX2}
Let  $f_1(z) = f_0(z) (z-e^{i\theta})$, where  $f_0$   is the  function from
the preceding example. Also let $\gamma$ be the curve from the  same example.
Since for every $\Delta_r\gamma,\, r\in (0,1)$     there exist a    vicinity
$O(r)$ of the point $e^{i\theta}$ such that for every $z\in O(r)\cap\Delta_r
\gamma$ holds
\begin{equation*}
|f_1(z)| =|f_0(z)| |z-e^{i\theta}|\le C(r)|z-e^{i\theta}|\rightarrow 0 \quad
  \text{as}\quad z\rightarrow e^{i\theta},
\end{equation*}
we have that  $0$ is $\Delta_\gamma-$boundary value  of $f$.    On the other
side, any $P-$sequence $\{z_k\}$ for   the function $f_0$ is a  $P-$sequence
for  $f_1$,   what may be proved in the same way as for $f_0$. That sequence
is contained  in $\mathcal{A}$. Since    $\lim_{k\rightarrow\infty}f_1(z_k)=
\infty$,     it follows that  $\infty,\, 0\in C(f,\mathcal{A}, e^{i\theta})$.
\end{example}

The example of function $f_1$ shows that in general case does not exist    a
set  which contains all  $\Delta_r\gamma,\, r\in (0,1)$ such that along this
set the function has the unique cluster point, i.e., Example \ref{EX2} shows
that the theorem on the existence   of curvilinear boundary values cannot be
improved  in  the     direction which means the  expansion of sets $\Delta_r
\gamma,\,  r\in (0,1)$.

\section{Applications}
Let   $\gamma\subseteq\mathbb{D}$ be a curve  which terminates in a    point
$e^{i\theta}\in\Gamma$  and   which  is    tangent on the cycle $\Gamma$  in
that point. Denote by $\Delta_{\alpha,\rho}\gamma,\, \rho\in (0,1),\, \alpha
\in(0,\pi)$ the sub--domain of $\mathbb{D}$ bounded  by  $\gamma$, the chord
$h(\theta ,\alpha),\, \alpha\in (0,\frac\pi 2)$ of $\mathbb{D}$  and  by the
arc        of $D_\rho = \{z : |z - e^{i\theta} | =\rho \},\, \rho\in (0,1)$.
Moreover, denote
\begin{equation*}
G^\theta_{\gamma,r,\alpha,\rho} = \Delta_r\gamma  \cup\Delta_{\alpha,\rho}
\gamma,\quad   r,\, \rho \in (0,1),\, \alpha\in (0,\pi)
\end{equation*}
(see Figure 2).  It  is easy to check that  $\Delta_{\alpha,\rho}\gamma\not
\subseteq G^\theta _{\gamma, r,\alpha, \rho}$ for all $r,\, \rho\in (0,1),\,
\alpha\in (0,\pi)$.
\begin{figure}[htp]\label{FIG2}
\centering
\includegraphics{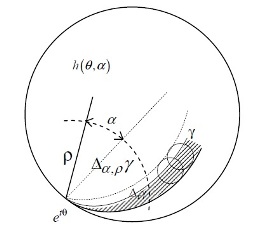}
\caption{}
\end{figure}

Lehto     and Virtanen (see  Remark, p. 53 in \cite{LEHTO.VIRTANEN.ACTA}, or
Remark on the  page 124 in \cite{LOHWATER}) showed that a normal meromorphic
function $f$  in the disc $\mathbb{D}$   which in a point $e^{i\theta}$ has
an asymptotic boundary value  $\lim_{\gamma\ni z\rightarrow e^{i\theta}}f(z)
=c\in\overline{\mathbb{C}}$     satisfies  $C(f, \Delta_{\alpha,\rho}\gamma,
e^{i\theta} )= \{c\}$ for all $\rho\in (0,1)$       and  $\alpha\in (0,\pi)$.
The following theorem shows that the result of Lehto   and  Virtanen for the
case   of a simple curve which is tangent on $\Gamma$   may be improved   in
the sense that  domains     $\Delta_{\alpha,\rho}\gamma,\,   \rho\in (0,1),\,
\alpha\in (0,\pi)$ along which an asymptotic value $c$  exists for  a normal
meromorphic          function in $\mathbb{D}$, may be replaced by the domain
$G^\theta_{\gamma,r,\alpha,\rho},\, r,\, \rho\in (0,1),\, \alpha\in (0,\pi)$.
Namely, we have the following

\begin{theorem}\label{TH17}
Let          $f:\mathbb{D}\rightarrow\overline{\mathbb{C}}$  be a   normal
meromorphic   function in the disc $\mathbb{D}$ i.e., $f\in N $    and let
$\lim_{\gamma\ni z\rightarrow e^{i\theta}} f(z) =c\in\overline{\mathbb{C}}$.
Then
\begin{equation*}
\bigcup _{r,\, \rho\in (0,1),\, \alpha\in (0,\pi)} C(f, G^\theta_{\gamma,r,\alpha,\rho},e^{i\theta})=\{c\}.
\end{equation*}
\end{theorem}

\begin{proof}
From the mentioned result of Lehto and Virtanen we have $C(f, \Delta_{\alpha,
\rho}\gamma, e^{i\theta} )  =\{c\}$   for    all     $\alpha\in (0,\pi)$ and
for all $\rho\in (0,1)$.  From Theorem  \ref{TH13} it  follows $C(f,\Delta_r
\gamma,e^{i\theta}) = \{c\}$  for all $r\in (0,1)$. Thus, for  all $r,\, \rho
\in (0,1)$ and $\alpha \in (0,\pi)$ we have $ C(f, G^\theta_{\gamma,r,\alpha,
\rho}, e^{i\theta})=\{c\}$; this  means  that the union of all          sets
$C(f,G^\theta_{\gamma,r,\alpha,\rho},e^{i\theta})$ where $r,\, \rho\in (0,1)\,
\alpha\in (0,\pi)$ is equal to $\{c\}$.
\end{proof}

For simplicity in  what follows    we will   assume that any curve $\gamma$
that appears  is a simple  curve which connects the center  of $\mathbb{D}$
and  some point $e^{i\theta}\in\Gamma$ such  and any circle    $\Gamma_r  =
\{z:|z| = r\},\ 0<r<1$ intersects  in  exactly one point.

\v{S}aginjan \cite{SAGINJAN} proved the following  statement of uniqueness:
Let $f(z)$ be an analytic function  in the disc  $\mathbb{D},\, |f(z)|<1,\,
z\in \mathbb{D}$, and let  $f$   along a curve $\gamma$ satisfies      the
following estimate
\begin{equation}\label{EQ2}
|f(z)|\le \exp\left\{ - \frac { p(1-|z|)} {1-|z|}\right\},\quad z\in\gamma,
\end{equation}
where $p(t)$ is a function which arbitrary slow increase  to  $+\infty$ as
$t\rightarrow +0$,   then    $f(z)\equiv 0$     (see Theorem 2,  p. 23, in
\cite{SAGINJAN}).   The analytic    function   in $\mathbb{D}$    given by
$f(z) = \exp\{-\frac 1{1-z}\}$ shows that the condition  cannot be relaxed.
Gavrilov proved a theorem which is an analog of  the preceding   result of
\v{S}aginjan:  Let $f(z)$ be normal meromorphic in   $\mathbb{ D}$ and let
$\varepsilon$ be any positive number.  If  $f(z)$ along the radius $\arg z
= 0$ satisfies the inequality
\begin{equation*}
|f(z)|\le \exp\left \{ - \frac 1{(1-|z|)^{1+\varepsilon}}\right \},
\end{equation*}
then $f(z)\equiv0$ (see Theorem 1 and 2, pp. 4--6, in \cite{GAVRILOV.SAGINJAN}).
Another results which are generalizations  of the theorem of \v{S}aginjan
may be found  in \cite{EIDERMAN.ESSEN},    \cite{KHAVINSON.RUSS.MATH.SURV},
\cite{VUORINEN.MATH.Z}, \cite{DOLZHENKO}; see also \cite{KEG1}, \cite{KEG2}.

Let $P=\{p(t):t\in (0, b),\,  p(t)\uparrow +\infty\ \text{as}\ t\rightarrow
0^+\}$. It is easy to check that if $p(t)\in P$, then  $p_1(t)  =     c_1
p(c_2t^\varepsilon)\in P$,   where $c_1,\, c_2$ and $\varepsilon$ are any
positive  numbers.

\begin{lemma}\label{LE5}
Let $f:\mathbb{D}\rightarrow\overline{\mathbb{C}}$ be a meromorphic function
which    is   normal along a curve $\gamma$  in $D(r)$ for some $r\in (0,1)$.
Furthermore, assume
\begin{equation}\label{EQ3}
|f(z)|\le \exp\left\{-\frac{p(1-|\varphi(z)|)}{1-|\varphi(z)|}\right\},\quad z\in\gamma
\end{equation}
where $\varphi:\Delta_{r_1}\gamma\rightarrow\mathbb{D},\, r_1\in (0,r)$ is a
conformal mapping of the curvilinear    angle $\Delta_{r_1}\gamma $ onto the
disk $\mathbb{D}$   such that   $\varphi (0) = 0$,  and let $p(t)$    be  an
arbitrary function which belong to the class $P$. Then $f(z)\equiv0$.
\end{lemma}

\begin{proof}
From Theorem  \ref{TH15} and \eqref{EQ3}  it follows
\begin{equation*}
C(f, \overline{\Delta_{r_1}\gamma }, e^{i\theta}) =\{0\}.
\end{equation*}
Without lost of generality, we may assume that $|f(z )| <1$ for all $z\in
\Delta_{r_1}\gamma$. Let $\varphi:\Delta_{r_1}\gamma\rightarrow\mathbb{D}$
be a conformal mapping of $\Delta_{r_1}\gamma$ onto the  disk $\mathbb{D}$.
Then $F(\omega) = f\circ \varphi^{-1}(\omega ),\, \omega\in \mathbb{D}$ is
an  analytic        function which satisfies $|F (\omega)| <1,\, \omega\in
\mathbb{D}$. Since $z =\varphi^{-1}(\omega)$, from  \eqref{EQ3}  we obtain
\begin{equation*}
|f(\varphi^{-1}(\omega))|\le \exp\left\{-\frac{p(1-|\varphi(\varphi^{-1}(\omega))|)}{1-|\varphi(\varphi^{-1}(\omega))|}\right\},\quad \omega\in\gamma_1=\varphi(\gamma),
\end{equation*}
i.e.,
\begin{equation}\label{EQ4}
|F(\omega)|\le \exp\left\{-\frac{p(1-|\omega|)}{1-|\omega|}\right\},\quad \omega\in\gamma_1=\varphi(\gamma),
\end{equation}
where  $\gamma_1$ is a curve withe endpoint $e^{i\theta}$. From \eqref{EQ4}
and the  theorem of  \v{S}aginjan   we have  $ F(w)\equiv 0$
in $\mathbb{D}$.       This implies  $f (z)\equiv 0$ in  $\Delta_r \gamma
\subseteq\mathbb{D}$ and according to the classical theorem of uniqueness
for meromorphic functions     we conclude $f (z)\equiv 0$ in $\mathbb{D}$.
\end{proof}

For $0<\alpha < \frac \pi2$ denote
\begin{equation*}
A(e^{i\theta},\alpha,\rho,z) = \{ z\in\mathbb{D}:  | \arg{ (e^{i\theta} - z )}| <\alpha,\, |e^{i\theta} - z|<\rho \},
\end{equation*}
where
\begin{equation*}
\rho = \left\{
\begin{array}{ll}
1,             & \hbox{$0<\alpha\le \frac\pi3$},\\
2\cos\alpha, & \hbox{$\frac \pi3 < \alpha < \frac \pi2$}.
\end{array}
\right.
\end{equation*}
Then $A(e^{i\theta},\alpha,\rho,z)$  is the Stolz angle in $\mathbb{D}$ with
the vertex in   $e^{i\theta}$  and with  angle $2\alpha$. We have    $\alpha
\rightarrow \frac \pi2$  as  $\rho\rightarrow0$.

Furthermore, let  $[a,b]$  denote a segment       in the complex plane  with
endpoints $a,\,  b\in\mathbb{C}$.

Denote $\varphi_1(z) = -z,\, \varphi_2(z) = \rho^{-1} (1+z),\, \varphi_3(z)=
e^{i\alpha} z,\, \varphi_4(z) = z^{\frac \pi{2\alpha}}$,  let  $\varphi_5(z)
 =  \frac 12 (z+\frac 1z )$ be the function of Zhukovsky,  $\varphi_6(z)  =
z e^{-\pi i }$, and $\varphi_7(z) =\frac{z-i}{z+i}$. If $w=\varphi_\alpha(z)
=  \varphi _7\circ\cdots\circ \varphi_1(z)$, then      $\varphi_\alpha$ is a
conformal mapping of the Stolz angle $A(1,\alpha,\rho,z)$ onto $\mathbb{D}$,
and
\begin{equation*}
w=\varphi_\alpha (z) = 1-  \frac {4\rho ^{\frac \pi{2\alpha}} (1-z)^{\frac \pi{2\alpha}}}
{\left[(1-z)^{\frac \pi{2\alpha}}-\rho^{\frac \pi{2\alpha}}  \right]^2 +2\rho ^{\frac \pi{\alpha}} };
\end{equation*}
moreover, $\psi (A(e^{i\theta},\alpha,\rho,z))  = \mathbb{D},\, \psi(z)  =
\varphi_\alpha(e^{-i\theta}z)$.

In the sequel we will simply write $\varphi$ instead of $\varphi_\alpha$, where
$\alpha\in (0,\frac\pi2)$ is fixed. We have $\varphi ^{-1} (A(1,\beta, r,w))
\subseteq A(1,\alpha,\rho,z),\, \beta\in (0,\frac \pi2),\, r\in (0,1),\, \varphi^{-1}([-1,1])
= [1-\rho ,1],\, \varphi^{-1}(-1) = 1-\rho,\,  \varphi ^{-1} (1) = 1$.

We formulate our theorems for  Stolz angles  $A(1,\alpha,\rho,z)$.

\begin{lemma}\label{LE6}
For all fixed  $\alpha,\, \beta\in (0,\frac\pi2)$ there exist constants $m=
m(\alpha, \beta )>0$  and   $M = M (\alpha,\beta)>0$    such that  for all
$\omega\in A(1,\beta , \rho,w)$ we  have
\begin{equation}\label{EQ5}
m(1-|z|)^{\frac \pi{2\alpha}}\le 1-|\omega|\le M (1-|z|)^{\frac\pi{2\alpha}},
\end{equation}
where $z=\varphi^{-1}(\omega)$.
\end{lemma}

\begin{proof}
Since  for all $z\in A(1,\alpha,\rho,z)$ and $\omega\in  A(1,\beta,\rho,w)$
we have
\begin{equation*}
1-\omega= \frac {4\rho ^{\frac \pi{2\alpha}} (1-z)^{\frac \pi{2\alpha}}}{ \left[(1-z)^{\frac\pi{2\alpha}}
-\rho ^{\frac \pi{2\alpha}}\right]^2+2\rho^{\frac\pi{\alpha}} }
\end{equation*}
and
\begin{equation*}
|1-\omega|<\frac{2}{\cos \beta} (1-|\omega|) = c(1-|\omega|),\quad c= \frac 2 {\cos \beta},
\end{equation*}
it follows that for all $\omega\in  A(1,\beta,\rho,w)$          and  $z =
\varphi^{-1}(\omega)$  we have
\begin{equation}\label{EQ6}
\begin{split}
&\left|\frac {4\rho ^{\frac{\pi}{2\alpha}} (1-z)^{\frac \pi{2\alpha}}}
{ \left[(1-z)^{\frac \pi{2\alpha}} -\rho ^{\frac \pi{2\alpha}}\right]^2+2\rho^{\frac \pi \alpha} }\right|
\frac{\cos\beta}2 (1-|z|)^{\frac \pi{2\alpha}}\le 1-|\omega|\le
\\&\left|\frac {4\rho ^{\frac \pi{2\alpha}} (1-z)^{\frac \pi{2\alpha}}}
{\left[(1-z)^{\frac \pi{2\alpha}} -\rho^{\frac \pi{2\alpha}}\right]^2+2\rho^{\frac \pi\alpha}}\right|
\left(\frac 2{\cos\alpha}\right)^{\frac \pi{2\alpha}} (1-|z|)^{\frac \pi{2\alpha}}.
\end{split}
\end{equation}
Since  $\alpha$   and $\beta$ are fixed,    $r=2\cos\alpha$ is also  fixed.
Since
\begin{equation*}
\phi(z)=\left|\frac {4\rho ^{\frac \pi{2\alpha}} (1-z)^{\frac \pi{2\alpha}}}{ \left[(1-z)^{\frac \pi{2\alpha}} -\rho ^{\frac\pi{2\alpha}}\right]^2+2\rho^{\frac\pi\alpha} }\right|
\end{equation*}
is a  continuous function  in  $A(1,\alpha,\rho,z)$, the function $f(z)$ on
the  compact  set $\overline{\varphi^{-1}(A(1,\beta,\rho,w))}$     achieves
its minimum and maximum; let $\phi_{\rm min}(z) = c_1 (\alpha,\beta)=c_1>0$
and $\phi_{\rm max} ( z) =  c _2(\alpha,\beta)=c_2<\infty$. From \eqref{EQ6}
we  obtain  that for all     $\omega\in     A(1,\beta,\rho,w)$    and  $z =
\varphi ^{-1} (\omega)$ holds
\begin{equation*}
m (1-|z|)^{\frac \pi{2\alpha}}\le 1-|\omega|\le M (1-|z|)^{\frac \pi{2\alpha}},
\end{equation*}
where we have denoted   $m=\frac {\cos \beta}2c_1>0$ and $M = \left(\frac 2
{\cos \alpha}\right)^{\frac\pi{2\alpha}}c_2<\infty$.
\end{proof}

\begin{theorem}\label{TH18}
Let  $\gamma$    be a simple        curve   in the disc $\mathbb{D}$ with one
endpoint  in $1$ which  is not tangent to  $\Gamma$.        Let $f:\mathbb{D}
\rightarrow \overline{\mathbb{C}}$        be a meromorphic function normal in
$\mathbb{D}$  along  the curve $\gamma$. If for all $\alpha\in (0,\frac\pi2)$
and  $p_1\in P$ holds
\begin{equation}\label{EQ7}
|f(z)|\le \exp\left\{ - \frac{p_1(1-|z|)}{(1-|z|)^{\frac\pi{2\alpha}}}  \right\},\quad z\in[\rho,1],
\end{equation}
where
\begin{equation*}
\rho = \left\{
\begin{array}{ll}
1,             & \hbox{$0<\alpha\le\frac \pi3$},\\
2\cos\alpha, & \hbox{$\frac\pi 3 < \alpha < \frac\pi2$},
\end{array}
\right.
\end{equation*}
then $f (z)\equiv 0$.
\end{theorem}

\begin{proof}
From     assumptions    of this theorem  and  Theorem \ref{TH7}  we obtain
that $f$ is normal  in $\mathbb{D}$ along the  radius  of the disc   which
terminates in $1$, i.e., $f$ is normal in     $\mathbb{D}$ along any curve
$[a,1],\, 0<a<1$. From \eqref{EQ7} we have that $0$ is an  angular boundary
value for  $f$.

Let
\begin{equation*}
w = \varphi (z) =
\frac {4\rho ^{\frac \pi{2\alpha}} (1-z)^{\frac \pi{2\alpha}}}{ \left[(1-z)^{\frac\pi{2\alpha}}
 -\rho ^{\frac\pi{2\alpha}}\right]^2+2\rho^{\frac\pi\alpha} }
\end{equation*}
be a conformal     mapping which maps the Stolz angle $A(1,\alpha,\rho,z)$
onto   the disk $\mathbb{D}$.

For any fixed   $\beta\in (0,\frac \pi2)$ and $r=1-\rho$ from  inequality
\eqref{EQ5} and Lemma  \ref{LE6} we obtain that  for all  $\omega\in A(1,\beta,\rho,w)$
holds
\begin{equation}\label{EQ8}
m(1- |z| )^{\frac \pi{2\alpha}}\le 1-|\omega|, \quad z=\varphi^{-1}(\omega).
\end{equation}
It follows that this  inequality holds for $\omega\in [1-\rho ,1]$    and
$z =\varphi ^{-1}(\omega)\in [\rho ,1].$  From \eqref{EQ8} we obtain
\begin{equation*}
p(1- |\varphi (z)| )\le p (m (1-|z|)^{\frac\pi{2\alpha}} )
\end{equation*}
for all $ w\in [1-\rho ,1],\, z =\varphi^{-1} (w)\in[\rho ,1]$ and $p\in P$.
Further, we have
\begin{equation*}
\frac{p(1-|\varphi(z)|)}{1-|\varphi(z)|}\le \frac{p(m(1-|z|)^{\frac \pi{2\alpha}})}{m(1-|z|)^{\frac\pi{2\alpha}}}.
\end{equation*}
Hence
\begin{equation}\label{EQ9}
-\frac{p(1-|\varphi(z)|)}{1-|\varphi(z)|}\ge -\frac {p_1(1-|z|)}{(1-|z|)^{\frac\pi{2\alpha}}}
\end{equation}
for $z\in\gamma_1   =   [\rho,1]$, where we have denoted $p_1(t)= m^{-1} p
(m t^{\frac \pi{2\alpha}})\in P$.  From  \eqref{EQ7}  and  \eqref{EQ9}  we
obtain
\begin{equation*}
|f(z)|\le \exp\left\{ - \frac {p(1-|\varphi(z)|)}{1-|\varphi(z)|} \right\},\quad z\in\gamma_1=[\rho,1].
\end{equation*}
From Lemma  \ref{LE5}  it follows $f(z)\equiv  0$.
\end{proof}

\begin{remark}\label{RE8}
From  the  proof of Theorem  \ref{TH18}  we see that this theorem holds if
instead of  normality in $\mathbb{D}$  of $f$   along   the curve $\gamma$
we have normality of  $f$ along $\gamma_1 = [a,1]$ in  $D_r$, where $r,\,a
\in (0,1)$. Then in the inequality \eqref{EQ7} for the angle $\alpha$   we
have $0<\alpha <\frac \pi2-\arctan \frac{1-r^2}{2r}$.     Namely, from the
preceding conditions    it follows that the function $f$, along the domain
which is bounded by two horo--cycles that contain $\pm 1$ and $\pm ri$ and
the circle $\{z : |z -1| = 1\}$ has a cluster set  which contains only $0$.
Since       for $0<\alpha<\frac\pi 2- \arctan \frac {1-r^2}{2r}$ the angle
$A(1,\alpha,\rho,z)$ is the subset of  this domain, we obtain $C(f,1, A(1,\alpha,\rho,z ))=\{0\}$,
hence a proof in this case goes in the same way as  the proof of   Theorem
\ref{TH18}.
\end{remark}

\begin{theorem}\label{TH19}
Let     $f:\mathbb{D}\rightarrow\overline{\mathbb{C}}$   be a meromorphic
function  normal in  $D(r)$,  where $r\in(0,1)$, along $\gamma_1 = [a,1],\,
0 < a <1$.  Let,  moreover,
\begin{equation*}
|f(z)| \le \exp \left\{ -\frac 1{(1-|z|)^k}\right\},\quad z\in[a,1],
\end{equation*}
where $k>\frac \pi {2\alpha},\, 0<\alpha<\frac\pi2-\arctan \frac{1-r^2}{2r}$.
Then $f\equiv 0$.
\end{theorem}

Theorem \ref{TH19}     follows immediately from Theorem  \ref{TH18} and the
inequality  $t^q<t^p$, if $0 < t <1$  and $q < p$.

\begin{remark}\label{RE9}
Theorem     \ref{TH19} is actually the result of Gavrilov; see Theorem 2 in
\cite{GAVRILOV.SAGINJAN}.
\end{remark}

\begin{theorem}\label{TH20}
Let $\gamma$ be a curve in $\mathbb{ D}$ which terminates in $1$ and  which
is not tangent  to    $\Gamma$ in this point. Let $f :\mathbb{D}\rightarrow
\overline{\mathbb{C}}$ be a meromorphic     function normal  in $\mathbb{D}$
along $\gamma$. If for all integers $n\ge 1$ holds
\begin{equation}\label{EQ10}
|f(z)|\le  \exp\left \{-\frac 1 {(1-|z|)^{1+\frac 1n}} \right\},\quad z\in[a,1],
\end{equation}
then $f (z) \equiv 0$.
\end{theorem}

\begin{proof}
Since   $0<\alpha<\frac\pi2$   and  $p_1,\ p_2\in P$ are arbitrary in Theorem
\ref{TH18}, if we set $\alpha=\frac n{2n+1}\pi$ and $p_1(t) = t^{-\frac1{2n}}$,
where $n\ge 1$ is an integer,    inequality    \eqref{EQ7}  take  the form
\eqref{EQ10}; from Theorem  \ref{TH18} we have $f (z ) \equiv 0$.
\end{proof}

\begin{remark}\label{RE10}
According to the Archimedean  principe for  real numbers on     inequality
\eqref{EQ10}   and  Theorem \ref{TH20}  we deduce the mentioned  result of
Gavrilov      (Theorem 1 in \cite{GAVRILOV.SAGINJAN}), but for meromorphic
functions  normal  in the  disc $D$ along a curve  which is not tangent to
$\Gamma$.
\end{remark}


\begin{thebibliography}{10}


%% A
\bibitem{AIKAWA.PAMS} %[37]
{H. Aikawa}, \emph{Harmonic functions having no tangential limits},
Proc. Amer. Math. Soc. \textbf{108} (1990), 457--464.

\bibitem{AIKAWA.JLMS} %[38]
{H. Aikawa}, \emph{Harmonic functions and Green potentials having no tangential limits},
J. London Math. Soc. (2) \textbf{43} (1991), 125--136.

\bibitem{AIKAWA.COMPLEX} %[39,40]
{H. Aikawa}, \emph{Fatou and Littlewood theorems for Poisson integrals with respect to non--integrable kernels},
Complex Var. Theory Appl. \textbf{49} (2004), 511--528.


%% B
\bibitem{BAGEMIHL.SEIDEL.AASF} %[12]
{F. Bagemihl and W. Seidel}, \emph{Sequential and continuous limits of meromorphic functions},
Ann. Acad. Sci. Fenn. Ser. A I \textbf{280} (1) (1960), 1--17.

\bibitem{BERBERIAN} %[44]
{S. L. Berberian}, \emph{Boundary behavior of subharmonic and meromorphic functions},
Dissertation, Moscow State University, Moscow, 1979 (Russian).


%% C
\bibitem{CASORATI} %[50]
{F. Casorati}, \emph{Teorica delle funzioni di variabili complesse},
Pavia, 1868.

\bibitem{COLLINGWOOD.LOHWATOR.BOOK} %[7]
{E. F. Collingwood and A. J. Lohwater}, \emph{The Theory of Cluster Sets},
Cambridge University Press, 1966.


%% D
\bibitem{DEVYATKOV} %[48]
{A.P. Devyatkov}, \emph{Limit sets and boundary properties of removable singularities of sequences of functions},
Dissertation, Tyumen, 2008 (Russian).

\bibitem{DOLZHENKO} %[56]
{E. P. Dolzhenko}, \emph{Boundary--value uniqueness theorems for analytic functions},
Math. Notes \textbf{25} (1979), 437--442.

\bibitem{DRAGOSH.J.ANALYSE} %[30]
{S. Dragosh}, \emph{Horocyclic boundary behavior of meromorphic functions},
J. d'Analyse Math. \textbf{22} (1969), 37--48.


%% E
\bibitem{EIDERMAN.ESSEN} %[53]
{V. Ya. \`{E}iderman and M. Ess\'{e}n}, \emph{Uniqueness theorems for analytic and subharmonic functions},
St. Petersburg Math. J. \textbf{14} (6) (2003), 889--952.

\bibitem{EVGRAFOV.BOOK} %[28]
{M. A. Evgrafov}, \emph{Analytic functions},
2nd ed., ''Nauka'', Moscow, 1968 (Russian).


%% F
\bibitem{FUKS} %[15]
{B. A. Fuks}, \emph{Non--Euclidean geometry in the theory of conformal and quasi--conformal mappings},
Moscow, 1951 (Russian).


%% G
\bibitem{GAUTHIER.NAGOYA} %[24]
{P. M. Gauthier}, \emph{A criterion for normalcy},
Nagoya Math. J. \textbf{32} (1968), 277--282.

\bibitem{GAUTHIER.CANADIAN.J.MATH} %[25]
{P. M. Gauthier}, \emph{{\it Cercles de remplissage} and asymptotic behaviour along circuitous paths},
Canad. J. Math. \textbf{22} (1970), 389--393

\bibitem{GAVRILOV.SBORNIK.67} %[19]
{V. I. Gavrilov}, \emph{On the distribution of values of functions meromorphic in the unit disk witch are not normal},
Math. Sbor. \textbf{67} (1965), 408--427 (Russian).

\bibitem{GAVRILOV.SBORNIK.71} % [22]
{V. I. Gavrilov}, \emph{Meromorphic functions in the unit disk with the prescribed growth of the spherical derivative},
Math. Sbor. \textbf{71} (1966), 386--404 (Russian).

\bibitem{GAVRILOV.SAGINJAN} %[26] [29]
{V. I. Gavrilov}, \emph{On a certain theorem of A. L. \v{S}aginjan},
Vestnik Moskov. Univ. Ser. I Mat. Meh. \textbf{21} (2) (1966), 3--10 (Russian).

\bibitem{GAVRILOV.BURKOVA.DOKL}  %[16]
{V. I. Gavrilov and E. F. Burkova},
\emph{On meromorphic functions generating normal families on subgroups of conformal automorphisms of the unit disk},
Dokl. Akad. Nauk SSSR \textbf{245} (1979), 1293--1296 (Russian).

\bibitem{GAVRILOV.ZAHARIAN} %[35]
{V. I. Gavrilov and V. S. Zaharian}, \emph{Lindelof set of points of arbitrary complex functions},
Akad. Nauk Armyan. SSR Dokl. \textbf{86} (1) (1988), 12--16 (Russian).


%% H
\bibitem{HASSAN} %[34]
{A. A.--R. Hassan}, \emph{On cluster sets of functions along tangential paths},
Dissertation, Moscow State University, Moscow, 1984 (Russian).

\bibitem{HAYMAN.BOOK} %[14]
{W. K. Hayman}, \emph{Meromorphic functions},
Oxford Mathematical Monographs, Clarendon Press, Oxford, 1964.


%% K
\bibitem{KEG1} %[57]
{\`{E}. M. Kegejan}, \emph{On the behaviour of an analytic function near the boundary of the region},
Akad. Nauk Armyan. SSR Dokl. \textbf{36} (1963), 263--269 (Russian).

\bibitem{KEG2} %[58]
{\`{E}. M. Kegejan}, \emph{On the radial behavior of functions analytic in a circle},
Akad. Nauk Armyan. SSR Dokl. \textbf{37} (1963), 241--247 (Russian).

\bibitem{KHAVINSON.RUSS.MATH.SURV} %[54]
{S. Ya. Khavinson}, \emph{Extremal problems for bounded analytic functions with interior side conditions},
Russ. Math. Surv. \textbf{18} (2) (1963), 23--96 (Russian).


%% L
\bibitem{LAPPAN.MATH.Z} %[45]
{P. Lappan}, \emph{Some results on harmonic normal functions},
Math. Z. \textbf{90} (1965), 155--159.

\bibitem{LEHTO.VIRTANEN.ACTA} %[2]
{O. Lehto and K. I. Virtanen}, \emph{Boundary behaviour and normal meromorphic functions},
Acta Math. \textbf{97} (1957), 47--65.

\bibitem{LINDELOF} %[1]
{E. L. Lindel\"{o}f},
\emph{Sur un principe g\'{e}n\'{e}ral de l'analyse et ses applications \`{a} la th\'{e}orie ed la repr\'{e}sentation conforme},
Acta Soc. Sci. Fenn. \textbf{46} (4) (1915), 1--35.

\bibitem{LOHWATER} %[5]
{A. Lohwater}, \emph{The boundary behavior of analytic functions},
Itogi Nauki Tekh., Ser. Mat. Anal. \textbf{10} (1973), 99--259 (Russian).

\bibitem{LOVSHINA.DOKL} %[17]
{G. D. Lovshina}, \emph{On meromorphic functions normal and hypernormal with respect to a subgroup of the unit disk},
Dokl. Akad. Nauk SSSR \textbf{252} (1980), 438--440 (Russian).


%% M
\bibitem{MEEK.PAMS} %[41]
{J. L. Meek}, \emph{Subharmonic versions of Fatou's theorem},
Proc. Amer. Math. Soc. \textbf{30} (1971), 313--317.

%\bibitem{LAVERDE} %[47]
%{J. E. Mejia L.}, \emph{Boundary properties of equimorphic functions},
%Dissertation, Moscow State University, Moscow, 1981 (Russian).

\bibitem{LAVERDE} %[47]
{J. E. Mejia L.}, \emph{Boundary properties of equimorphic functions},
Dokl. Akad. Nauk SSSR \textbf{265} (1982), 35--38 (Russian).

\bibitem{} %[33]
{M. M. Mirzojan},
\emph{The characteristic of boundary peculiarities of meromorphic functions and functions that are generated by cluster sets along tangential directions},
Akad. Nauk Armyan. SSR Dokl. \textbf{66} (5) (1978), 263--266 (Russian).

\bibitem{MONTEL.BOOK} %[9]
{P. Montel}, \emph{Lecons sur les familles normales de fonctions analytiques et leurs applications},
Gauthier--Villars, Paris, 1927.


%% N
\bibitem{NOSHIRO.BOOOK} %[6]
{K. Noshiro}, \emph{Cluster sets},
Springer--Verlag, Berlin, Gottingen, Heidelberg, 1960.

\bibitem{NOSHIRO.HOKKAIDO} %[11]
{K. Noshiro}, \emph{Contributions to the theory of meromorphic functions in the unit circle},
J. Fac. Sci. Hokkaido Univ. \textbf{7} (1939), 149--159.


%% O
\bibitem{OSTROWSKI.MATH.Z} %[20]
{A. Ostrowski}, \emph{\"{U}ber Folgen analytischer Funktionen und einige Versch\"{a}rfungen des Picardschen Satzes},
Math. Z. \textbf{24} (1926), 215--258.


%% P
\bibitem{PAVICEVIC.NEW.ZELAND} %[21]
{\v{Z}. Pavi\'{c}evi\'{c}}, \emph{Meromorphic functions generating normal families in arbitary open subset of the unit disk},
New Zealand J. Math. \textbf{28} (1999), 89--106.

\bibitem{PAVICEVIC.MOSCOW.BULL} %[23]
{\v{Z}. Pavi\'{c}evi\'{c}}, \emph{Normality and curvilinear limits of meromorphic functions and their applications},
Moscow University Math. Bull. \textbf{66} (4) (2011), 171--172

\bibitem{PAVICEVIC.SUSIC.DOKL}  %  [18]
{\v{Z}. Pavi\'{c}evi\'{c} and J. \v{S}u\v{s}i\'{c}},
\emph{An application of cycles properties of dynamical systems to the investigations of  boundary  limits of arbitrary  functions},
Dokl. Ross. Akad. Nauk \textbf{387} (2002), 16--18 (Russian).


%% R
\bibitem{RIPPON.JLMS} %[43]
{P. J. Rippon}, \emph{The boundary cluster sets of subharmonic functions},
J. London Math. Soc. (2) \textbf{17} (1978), 469--479.

\bibitem{RIIHENTAUS.ADV.ALG.ANAL} %[42]
{J. Riihentaus}, \emph{A weighted boundary limit result for subharmonic functions},
Adv. Algebra and Analysis \textbf{1} (2006), 27--38.

\bibitem{RUNG.MATH.Z} %[42]
{D. C. Rung}, \emph{Asymptotic values of normal subharmonic functions},
Math. Z. \textbf{84} (1964), 9--15.

\bibitem{RUNG.PACIFIC} %[32]
{D. C. Rung}, \emph{Meier type theorems for general boundary approach and $\sigma-$porous exceptional sets},
Pacific J. Math. \textbf{76} (1978), 201--213.


%% S
\bibitem{SCHIFF.BOOK} %[10]
{J. L. Schiff}, \emph{Normal families},
Springer--Verlag, New York, 1993.

\bibitem{SEIDEL.TAMS} %[3]
{W. Seidel}, \emph{On the cluster values of analytic functions},
Trans. Amer. Math. Soc. \textbf{34} (1932), 1--21.

\bibitem{SEIDEL.WALSH.TAMS} %[4]
{W. Seidel and J. L. Walsh}, \emph{On the derivatives of functions analytic in the unit circle and their radii of univalence and of $p-$valence}, Trans. Amer. Math. Soc. \textbf{52} (1942), 128--216.

\bibitem{SHABAT.BOOK} %[52]
{B. V. Shabat}, \emph{Introduction to Complex Analysis}, I,
''Nauka'', Moscow, 1985 (Russian).

\bibitem{SOKHOTSKY} %[49]
{Y. V. Sokhotsky}, \emph{The theory of integral residues with some applications},
St. Petersburg, 1868 (Russian).

\bibitem{SOLOMENCEV} %[36]
{E. D. Solomentsev}, \emph{An example of a bounded and continuous subharmonic function which does not have angular boundary values},
Siberian Math. J. \textbf{1} (1960), 488--491 (Russian).

\bibitem{STAYANARAYANA.WEISS.PAMS}  %[31]
{U. V. Satyanarayana and M. L. Weiss}, \emph{The geometry of convex curves tending to $1$ in the unit disc},
Proc. Amer. Math. Soc. \textbf{41} (1973), 159--166.

\bibitem{SAGINJAN} %[27]
{A. L. \v{S}aginjan}, \emph{On a  fundamental inequality in the theory of functions and some of its applications},
Akad. Nauk Armyan. SSR Dokl. Ser. Mat. \textbf{12} (1) (1959), 3--25 (Russian).

\bibitem{SUSIC.PAVICEVIC.BALCANICA} % [13]
{J. \v{S}u\v{s}i\'{c} and \v{Z}. Pavi\'{c}evi\'{c}}, \emph{The normality of meromorphic functions along an arbitrary curve},
Math. Balkanica \textbf{22} (2008), 121--131.


%% V
\bibitem{VALIRON} %[8]
{G. Valiron}, \emph{Fonctions analytiques},
Presses Univ. de France, Paris, 1954.

\bibitem{VAISALA.AASF} %[46]
{J. V\"{a}is\"{a}l\"{a}}, \emph{On normal quasiconformal functions},
Ann. Acad. Sci. Fenn. Ser. A I \textbf{266} (1959), 33pp.

\bibitem{VUORINEN.MATH.Z} %[55]
{M. Vuorinen}, \emph{Koebe arcs and quasiregular mappings},
Math. Z. \textbf{190} (1985), 95--106.


%% W
\bibitem{WEIERSTRASS} %[51]
{C. Weierstrass}, \emph{Zur Theorie der eindeutigen analytischen Funktionen},
Abhandlungen der Preussischen Akademie der Wissenschaften. Physikalisch-mathematische Klasse, 1876.


%% Y
\bibitem{YOSHIDA} % [33]
{H. Yoshida}, \emph{Tangential boundary behaviors of meromorphic functions in the unit disc},
J. Fac. Engrg. Chiba Univ. \textbf{25} (48) (1974), 91--98.


\end{thebibliography}
\end{document}